\documentclass[12pt,a4paper]{amsart}
\usepackage[top=1in, bottom=.9in, left=0.9in, right=0.9in]{geometry}
\usepackage[utf8]{inputenc}
\usepackage[T1]{fontenc}
\usepackage{amsmath}
\usepackage{amssymb}
\usepackage{graphicx}
\usepackage{float}
\usepackage{version}
\usepackage{csquotes}
\MakeOuterQuote{"}
\usepackage[english]{babel}
\usepackage{xcolor}
\usepackage[linktocpage=true,pagebackref=false, colorlinks = true, linkcolor=blue, bookmarks=true, citecolor=blue, urlcolor=blue]{hyperref}
\usepackage{listings}
\usepackage{tikz-cd}

\newcommand{\del}{\partial}
\newcommand{\delbar}{\bar{\partial}}
\DeclareMathOperator{\GL}{GL}
\DeclareMathOperator{\SL}{SL}
\DeclareMathOperator{\Ker}{Ker}
\DeclareMathOperator{\im}{Im}
\newcommand{\C}{\mathbb{C}}
\newcommand{\R}{\mathbb{R}}
\newcommand{\Z}{\mathbb{Z}}

\theoremstyle{plain}
\newtheorem{teorema}{Theorem} [section]
\newtheorem{proposizione}{Proposition} [section]
\newtheorem{corollario}{Corollary}[section]
\newtheorem{lemma}{Lemma}[section]
\theoremstyle{definition}
\newtheorem{definizione}{Definition}[section]
\newtheorem{esempio}{Example}[section]
\newtheorem{osservazione}{Remark}[section]

\numberwithin{equation}{section}

\title{Non-K\"ahler Special Lagrangian submanifolds and SYZ mirror symmetry}
\author{Tristan C. Collins, Francesca Lusetti, Adriano Tomassini}
\date{}
\address{Department of Mathematics, University of Toronto, 40 St. George St., Toronto, ON}
\email{tristanc@math.toronto.edu}
\address{Dipartimento di Matematica ``Tullio Levi-Civita'', Università degli Studi di Padova, Via Trieste, 63,
35121, Padova, Italy}
\email{francesca.lusetti@phd.unipd.it}
\address{Dipartimento di Scienze Matematiche, Fisiche e Informatiche\\
Unit\`{a} di Matematica e Informatica,
Universit\`{a} degli Studi di Parma\\
Parco Area delle Scienze 53/A, 43124 \\
Parma, Italy}
\email{adriano.tomassini@unipr.it}
\thanks{\newline 
The first author is supported in part by NSERC Discovery grant RGPIN-2024-518857. The second author is supported by the department of Mathematics "Tullio Levi-Civita" of the University of Padova. The third author is partially supported by the Project PRIN 2022 ``Real and Complex Manifolds: Geometry and Holomorphic Dynamics 2022AP8HZ9'' and by GNSAGA of INdAM}

\begin{document}
	
	\begin{abstract}
	   We determine purely algebraic equations to identify \textit{SLags} generated by invariant distributions in a class of non-K\"ahler Calabi-Yau manifolds. We determine SLag distributions, determine which leaves integrate to compact submanifolds and study the deformation theory, which we find to be unobstructed. We apply our results to the Iwasawa manifold, the completely solvable 6-dimensional Nakamura manifold and the complex parallelizable Nakamura manifold. Through these examples we find families of topologically distinct \textit{SLags}, including the existence of SLag torus fibrations.  Following the proposal of Lau-Tseng-Yau, we compute the non-K\"ahler SYZ mirrors of Nakamura manifolds, together with their refined symplectic Bott-Chern cohomologies.  As a consequence, we find the existence of semi-flat non-K\"ahler mirror pairs which are not diffeomorphic.
	\end{abstract}
	
	\maketitle
    \tableofcontents
	
	\section{Introduction}
	A central geometric object in geometry and theoretical physics is given by 3-dimensional \textit{Calabi-Yau manifolds} namely, compact K\"ahler manifolds of complex dimension 3 with trivial canonical bundle.
	More recently, research has also been considering \textit{non-K\"ahler Calabi-Yau} manifolds i.e., compact complex non-K\"ahler manifolds with trivial canonical bundle. Following the works of Strominger \cite{Strominger} and Hull \cite{Hull}, non-K\"ahler Calabi-Yau 3-folds have come to play a central role in the study of string compactifications with flux. This subject has recently generated a great deal of interest in the mathematics literature; see, for instance,  \cite{CGPS}, \cite{CPY}, \cite{DBT}, \cite{DBT1} \cite{FLY}, \cite{Hitchin}, \cite{LTY}, \cite{PhongSurvey}, \cite{FPPZ21}, \cite{FPPZ22}, \cite{PPZ18}, \cite{AAG24}, \cite{GG25}, \cite{GGS26} and the references therein.\\[5pt]
	Given a non-K\"ahler Calabi-Yau manifold $ (M, J) $ of complex dimension $ 3 $, let $ \Omega $ be a $ (3,0) $-form which trivializes $ \Lambda^{3,0}(M) $. Fixing a $ J $-Hermitian metric on $ M $ with associated $ (1,1) $-form $\omega$, we can define oriented submanifolds of $ M $ called \textit{special Lagrangian submanifolds} (briefly, \textit{SLags}). Following Harvey-Lawson \cite{HL}, an oriented submanifold $ L $ is \textit{a SLag of phase} $\theta$ if the following conditions are satisfied:
	\begin{equation}
		\omega_{|_{L}} = 0 \qquad \text{Im}e^{-i\theta}\Omega_{|_{L}} = 0
	\end{equation}
	where $ \theta $ is a constant angle. SLags are concerned with several aspects of a non-K\"ahler Calabi-Yau manifold: they play a role in the topology of non-K\"ahler Calabi-Yau 3-folds and \textit{conifold transitions} (see \cite[theorem 1.1]{CGPS}, \cite{CPY}, \cite{FLY} and the references therein), they minimize a \textit{volume functional} in a given homology class and special Lagrangian torus fibrations are central to non-K\"ahler SYZ mirror symmetry \cite{LTY}.
	Given the definition, it is evident the connection with \textit{Lagrangian submanifolds} of symplectic manifolds $ (M, \omega) $, as well as \textit{special Lagrangian submanifolds} inside \textit{generalized Calabi-Yau manifolds} \cite{DBT}.\\[10pt]
	 One can ask several questions regarding the geometry of special Lagrangian submanifolds; in this work we focus on \textit{existence} and \textit{deformation theory}. Regarding existence, it is of interest to study which kind of SLags exist inside a given manifold, as well as their topology. As an example, in \cite[Section 3]{DBT} there has been constructed a non-K\"ahler generalized Calabi-Yau manifold foliated by special Lagrangian submanifolds.\\[5pt]
     About deformation theory, given a special Lagrangian submanifold $ L $ of a K\"ahler Calabi-Yau manifold it is a well-known Theorem of McLean \cite{Mclean} that the deformation theory is unobstructed, and first order deformations are characterized by 1-forms on $ L $ satisfying an elliptic equation $ \mathcal{L}\alpha = 0$. The well-known \textit{McLean's theorem} \cite{Mclean}, and the computations in \cite{Goldstein}, state that the kernel of $ \mathcal{L} $ is isomorphic to $ H^{1}(L;\R) $ and therefore first order deformations behave \textit{accordingly} to the first Betti number $ b_1(L) $. In the non-K\"ahler Calabi-Yau setting, the same description of first order deformations for SLags in terms of 1-forms still holds: in \cite[Lemma 2.2]{CGPS}, it is proved that $\alpha \in A^{1}(L)$ satisfies $ \mathcal{L}\alpha = 0 $ if and only if the following equations hold true
	 \begin{equation}\label{EquazioniDef}
	 	d\alpha + T\alpha = 0 \qquad d^{*}(|\Omega|\alpha) = 0
	 \end{equation}
	 where $ T $ is an operator defined via $ J $ and $\omega$. In particular, their application to the Iwasawa manifold $ \mathbb{I}_3 $ \cite[Section 3]{CGPS} shows a substantial difference with the symplectic setting; while the deformation theory is again unobstructed, the space of infinitesimal deformations is not determined solely by the topology of $L$.  This suggests that the deformation theory for special Lagrangians in non-K\"ahler Calabi-Yau manifolds may be subtle.  One of the goals of the current paper is to study the behavior of deformations by solving the equations \eqref{EquazioniDef}.\\[5pt]
     As we pointed out, special Lagrangian torus fibrations are the main ingredient to the non-K\"ahler approach to mirror symmetry proposed by Lau-Tseng-Yau \cite{LTY}. Indeed, a construction of semi-flat mirror pairs $(X, \omega, \Omega)$ and $(\check{X}, \check{\omega}, \check{\Omega})$ is obtained by dualizing fibrations in special Lagrangian tori, then showing the correspondence between the \textit{refined} versions of their \textit{Bott-Chern} \cite{BC} and \textit{symplectic Tseng-Yau} cohomologies \cite{TY} via the \textit{Fourier Mukai transform}.\\[10pt]
	 In this work we address these questions on non-K\"ahler Calabi-Yau 3-folds given by compact quotients $ \Gamma \backslash G $, where $ G $ is a Lie group and $\Gamma$ a discrete co-compact subgroup namely, a \textit{lattice}. It is well known that among this kind of quotients we find all \textit{nilmanifolds} and \textit{solvmanifolds} (i.e., quotients of type $ \Gamma \backslash G $ where $ G $ is a simply connected nilpotent resp., solvable Lie group).\\[5pt]
	 Starting from existence, our strategy for finding SLags is to consider distributions of vector fields on $ X $ which satisfy equations \eqref{DefSLag} and integrate to foliations whose leaves are SLags in $ X $. More precisely, we provide \textit{purely algebraic} equations whose evaluation directly tells us if the distribution can give foliations of SLags or not. These equations are presented in Lemma \ref{EqAlgebricheSLags} and from them we find four distributions (see the beginning of Section \ref{SpecialDistr}) which satisfy \eqref{DefSLag} for \textit{all} 3-folds as $ X $ (at least, after the correct convention on notations is fixed). What cannot be stated in general is integrability of the distributions or compactness of the integral leaves. In the examples of Section \ref{Examples} we address these aspects directly.\\[5pt]
	 Moving to first order deformations, we prove by a direct computation that the deformation equations \eqref{EquazioniDef} on $ X $ for SLags coming from the distributions of Section \ref{SpecialDistr} can be written in a simpler form: this is Theorem \ref{TeoremaEquazioniDeformazioni}. However, the solution of the system we find depends on the properties of both the manifold $ X $ and the SLag we are considering.\\[5pt]
     Finally, we compute explicitly the semi-flat mirror of the \textit{completely solvable} and \textit{complex parallelizable} Nakamura manifolds. For the completely solvable Nakamura manifold, our construction agrees at the Lie algebra level with the one presented in \cite{BV} and recently in \cite{CGKM} however, we build the mirror manifold following a geometrical/topological approach.
     Instead, the construction on the complex parallelizable Nakamura manifold is completely new and by the computation of the de Rham cohomology we find a non-diffeomorphic semi-flat mirror pair even if the geometric construction of both sides of the mirror is the same. This is the content of Theorems \ref{TheoremMirrorCPNakamura} and \ref{deRhamCohomCPNakamura}. What really changes is that complete solvability here disappears.
     These constructions also provide different ways of approaching the construction of different kinds of Nakamura manifolds, which also carry symplectic structures satisfying the \textit{Hard Lefschetz Condition} (see for example \cite{DBT1}, \cite{Kasuya1}, \cite{LT}). To our knowledge, these are the first examples of non-diffeomorphic mirror pairs of non-K\"ahler Calabi-Yau manifolds.\\[10pt]
	 The work is structured as follows. In Section \ref{Preliminaries} we recall definitions and results from the non-K\"ahler Calabi-Yau setting and fix notations. In Section \ref{Equations} we deal with the existence problem considering distributions of three vector fields on $ X $. We prove Lemma \ref{EqAlgebricheSLags} and consider the mentioned four special distributions in Proposition \ref{DistribuzioniSingole}. Section \ref{Deformations} is dedicated to first order deformations and the main result is Theorem \ref{TeoremaEquazioniDeformazioni}, in which we present the reduced form of deformations equations \eqref{EquazioniDef}. In Section \ref{Examples} we apply the previous results to the \textit{Iwasawa manifold} and to the \textit{completely solvable 6-dimensional Nakamura manifold}. On $ \mathbb{I}_3 $ we generalize the computations of \cite[Section 3]{CGPS}. Instead, the computations of the Nakamura manifold $ N $ are new and show the different behaviors of SLags, both in  topology and first order deformations. The main results of the section are contained in Theorem \ref{NumBettiSLagsCSNakamura}, Corollary \ref{EsistenzaSLagTopDiverse} and Theorem \ref{TopologyFoliations}, which show that inside the same non-K\"ahler Calabi-Yau manifold we can find SLags with completely different topological behaviors. In Section \ref{MirrorSymmCSNakamura} we address the construction of the mirror of the completely solvable Nakamura manifold, which is possible by the results contained in Theorem \ref{TopologyFoliations} of Section \ref{Examples}. After constructing the symplectic side of the mirror pair, we also compute its refined symplectic Bott-Chern cohomology. Finally, in Section \ref{MirrorSymmCPNakamura} we carry out the same constructions but on the \textit{complex parallelizable Nakamura manifold} and compare the cohomologies of the two sides of the mirror pair, proving Theorems \ref{TheoremMirrorCPNakamura} and \ref{deRhamCohomCPNakamura}.
     
    \section{Preliminaries}\label{Preliminaries}
    We start by recalling definitions and main results from the context of compact non-K\"ahler Calabi-Yau manifolds and the constructions of non-K\"ahler SYZ mirror symmetry. We fix some notations as well.
    \subsection{Non-K\"ahler Calabi-Yau manifolds and special Lagrangian submanifolds}
    We start by recalling the definition of compact non-K\"ahler Calabi-Yau manifolds. 
    \begin{definizione}[\cite{CGPS}]
        A $n$-dimensional compact \textit{non-K\"ahler Calabi-Yau manifold} $X$ is a complex (non-K\"ahler) manifold with trivial canonical bundle. Equivalently, $X$ admits a complex non-K\"ahler structure $J$ for which there exists a nowhere vanishing holomorphic $(n,0)$-form $\Omega$.
    \end{definizione} 
    This work will consider what we call non-K\"ahler Calabi-Yau \textit{3-folds} i.e, of complex dimension 3.
    \begin{definizione}\label{balancedMetric}
        Let $(X, J, \Omega)$ be a non-K\"ahler Calabi-Yau manifold and let $\omega$ be a real $(1,1)$-form defining an Hermitian metric on $X$. The metric $\omega$ is said to be \textit{balanced} if $d\omega^{n-1} = 0$.
    \end{definizione}
    In the case of a 3-fold, the balanced condition reads $d(\omega^2) = 0$. Given a non-K\"ahler Calabi-Yau manifold endowed with a metric $g$ and a $J$-compatible Hermitian metric $\omega$, $(X, J, \Omega, \omega)$, a function $|\Omega|_{g}$ (which can be interpreted as a norm) is defined by the equation
    \begin{equation}\label{NormOmega}
        |\Omega|^{2}_{g}\dfrac{\omega^n}{n!} = \dfrac{i^{n^2}}{2^n}\Omega \wedge \bar\Omega.
    \end{equation}
    \begin{osservazione}
        A structure $(X^3, \Omega, \omega)$ as we defined, with $(X^3, \Omega)$ non-K\"ahler Calabi-Yau manifold of complex dimension 3 and $\omega$ balanced has been considered in \cite{LTY}, where it is called a \textit{supersymmetric structure of type IIB}.
    \end{osservazione}
    \begin{osservazione}
        For the computations we will perform, we recall (see for example \cite{Strominger}, \cite[Equation (2.3)]{CGPS}) that the required supersymmetry condition on $\omega$ reads 
        \begin{equation}\label{conformalBalance}
            d(|\Omega|_{g}\omega^{n-1}) = 0.
        \end{equation}
       This condition is referred to as the \textit{conformally balanced} condition.
        In this work, the norm $|\Omega|_{g}$ will always be a constant therefore, the requirement is equivalent to definition \ref{balancedMetric}. Metrics $\omega$ satisfying \eqref{conformalBalance} have been observed to be critical points of a kind of energy functional;  we refer the reader to \cite{Michelsohn} for background on balanced metrics.
    \end{osservazione}
    Following Harvey-Lawson \cite{HL} we give the definition of \textit{special Lagrangian submanifold}.
    \begin{definizione}
        Let $(X, J, \Omega, \omega)$ be a (possibly non-K\"ahler) Calabi-Yau manifold, with the same notation as before. Let $\dim_{\R}X = 2n$. An oriented submanifold $L$ such that $\dim_{\R}L = n$ and the following equations
        \begin{equation}\label{DefSLag}
            \omega_{|_L} = 0 \qquad \text{Im}e^{-i\theta}\Omega_{|_{L}} = 0
        \end{equation}
         hold is called a \textit{special Lagrangian submanifold of angle $\theta$}.
    \end{definizione}
    \begin{osservazione}
        \begin{enumerate}
            \item The first equation $\omega_{|_L} = 0$ corresponds to the definition of \textit{Lagrangian submanifold} inside a symplectic manifold, with $\omega$ the symplectic form.
            \item The equations \eqref{DefSLag} imply (see \cite{HL}) that the form $(e^{-i\theta}|\Omega|_g^{-1})\Omega_{|_L}$ is a volume form for $L$. 
        \end{enumerate}
    \end{osservazione}
    The first question we address regards the \textit{existence} of SLags inside a given compact non-K\"ahler Calabi-Yau manifold. It is clear that finding SLags corresponds to find submanifolds $L$ satisfying both eqautions \eqref{DefSLag} however, this is not trivial in general. The method we follow is to look for SLags of dimension $3$ as the leaves of foliations arising from the integration of distributions of $3$ vector fields on $X$ satisfying \eqref{DefSLag}. In view of this, we recall that not all distributions can be integrated to submanifolds, only the \textit{involutive} ones (a distribution of vector fields $\{V_i\}_i$ on a manifold is called \textit{involutive} if $[V_i, V_j] $ still belongs to the distribution for every $i,j$). This is the well-known \textit{Frobenius' Theorem}. We also recall that \textit{compactness of the integral leaves is not ensured} by the Frobenius' Theorem or by other known results.\\[5pt]
    The second question we address is about deformation theory, precisely the \textit{first order deformations}.
    Let $(X, J,\Omega, \omega)$ be a compact non-K\"ahler Calabi-Yau manifold as fixed before and let $L \hookrightarrow X$ be a special Lagrangian submanifold. We can study the deformation of $L$ under the diffeomorphisms generated by the flow of vector fields normal to the image of $L$ inside $X$. In \cite[Lemma 2.2]{CGPS} it is proved the following result, which describes the space of these deformations for any given SLag $L$ (not necessarily compact).
    \begin{lemma}[\cite{CGPS}, Lemma 2.2]\label{LemmaDeformations}
        The infinitesimal deformation space of a special Lagrangian submanifold $L \subset (X, J,\Omega, \omega)$ is given by the space of smooth 1-forms $\alpha \in T^{*}L$ satisfying    
        \begin{equation}\label{EquazioniDefomazione}
		  d\alpha + T\alpha = 0 \qquad d(*\alpha) = 0
	    \end{equation}
        where the operator $T$ is defined by $T: \Omega^{1}(L) \rightarrow \Omega^2(L)$, $T\alpha = -d\omega(J\alpha*, .,.)$, being $\alpha^*$ the $g$-dual form of $\alpha$.
    \end{lemma}
    Lemma \ref{LemmaDeformations} completely characterizes the space of first order deformations through a system of differential equations.  However, this system is not intrinsic to $L$ thanks to the extrinsic operator $T$.  Using global geometric results in combination with some elliptic theory we will characterize the space of solutions for certain classes of special Lagrangians. In particular, the interest in study the dimension of the first order deformations space lies in understanding whether, given a SLag $L$ its deformations are \textit{unobstructed} i.e., are all first order deformations of $L$ integrable to genuine families SLag submanifolds. This corresponds exactly to being able to find a solution for equations \eqref{EquazioniDefomazione}.\\[5pt]
    It is interesting to compare the dimension of the first order deformations space with the first Betti number $b_1(L)$. In particular, if $\omega$ is a symplectic form, well-known results (\cite{Mclean}, \cite{Goldstein}) state that
    \[\dim\Ker\mathcal{L} = \dim H^1(L;\R) = b_1(L)\]
    where $\mathcal{L}$ is the elliptic operator whose kernel describes the deformations space and gives equations \eqref{EquazioniDefomazione}. This does not hold when $\omega$ is not closed, as the computations in \cite[Section 3]{CGPS} and in Section \ref{Examples} of this paper show. However, we can still compare the two spaces and see if the two dimensions are somehow still related; for example, is it true that the dimension of the first order deformation space is at most $b_1(L)$?

\subsection{Supersymmetric Type IIA and Type IIB $SU(n)$ structures}
Fix a real $2n$-dimensional manifold $X$. We start by recalling the definition of $SU(n)$ structure. We here mainly follow \cite{LTY} however, suitable references for the subject are also \cite{TV} or \cite{BV}.
\begin{definizione}\label{SUnStr}
An \textit{$SU(n)$ structure} on $X$ is a pair  $(\omega, \Omega)$ of differential forms satisfying the following conditions:
\begin{enumerate}
\item $\Omega$ is a nowhere-vanishing complex-valued $n$-form on $X$, inducing an almost complex structure $J$ on $X$ such $\Omega$ is an $(n,0)$ form with respect to this almost complex structure. 
\item $\omega$ is a non-degenerate real $(1,1)$-form with respect to $J$ such that $\omega(\cdot, J \cdot)$ defines a Hermitian metric on $X$. This in particular implies that $\Omega \wedge \omega = 0$.
\end{enumerate}
In particular, it holds that
$$ \Omega \wedge \bar{\Omega} = i^n \, F \cdot \frac{\omega^n}{n!} $$
for some nowhere-vanishing function $F$ on $X$, which is called the \textit{conformal factor} of the $SU(n)$ structure. 
\end{definizione}
\begin{osservazione}
\begin{enumerate}
    \item In the definition of $SU(n)$ structure it is not required that $d\Omega = 0$. However, it holds that $J$ is integrable (and therefore defines a complex structure) if and only if $d\Omega = 0$.
    \item Non-K\"ahler Calabi-Yau manifolds $(X, \Omega)$ with an Hermitian metric $\omega$ carry $SU(n)$ structures naturally.
\end{enumerate}
\end{osservazione}
When $n = 3$, we can impose further conditions on definition \ref{SUnStr}, which physically correspond to supersymmetry of type IIA and IIB.
\begin{definizione}
An $SU(3)$ structure $(X, \omega, \Omega)$ is said to be \textit{supersymmetric of type IIB} if $\Omega$ defines a complex structure and $\omega$ defines a balanced metric i.e.,  $d\Omega = 0$ and $d (\omega^{2}) = 0$.\\
An $SU(3)$ structure $(X, \omega, \Omega)$ is said to be \textit{supersymmetric of Type IIA} if $d \omega = 0$ and $d \mathrm{Re}\, \Omega = 0$.
\end{definizione}
There exist natural systems of equations satisfied by supersymmetric type IIA and IIB structures. Let $F$ be the conformal factor of the $SU(3)$ structure. We can define differential forms
$$\rho_B = 2i \partial \bar{\partial} \left(F^{-1} \cdot \, \omega \right) \qquad \rho_A = dd^{\Lambda} (F \cdot \mathrm{Im}\, \Omega)$$
where $d^\Lambda= d\Lambda - \Lambda d$ (and $\Lambda$ is the adjoint of the Lefschetz operator, see for example \cite{TY}). We get the following systems of equations for a supersymmetric $SU(3)$ structure of type IIB resp., type IIA:
\begin{equation}\label{SystemTypeIIB}
\begin{aligned}
d \Omega &= 0,\\
d(\omega^2) &= 0,\\
\Omega \wedge \bar{\Omega} &= -i \, F \cdot \frac{\omega^3}{6},\\
2i\partial \bar{\partial} \left(F^{-1} \cdot \omega\right) &= \rho_B.
\end{aligned}
\end{equation}
\begin{equation}\label{SystemTypeIIA}
\begin{aligned}
d \omega &= 0,\\
d (\mathrm{Re} \, \Omega) &= 0,\\
\Omega \wedge \bar{\Omega} &= -i \, F \cdot \frac{\omega^3}{6}, \\
dd^{\Lambda} (F \cdot \mathrm{Im}\, \Omega) &= \rho_A.
\end{aligned}
\end{equation}
More details about this type of structures can be find again in \cite{LTY}. Now we can see how these structures are the natural ones to consider in non-K\"ahler SYZ mirror symmetry.
    
\subsection{Non-K\"ahler SYZ mirror symmetry}
We recall the construction of \textit{non-K\"ahler semi-flat mirror pairs} (which follows the analogue well-known construction of the K\"ahler setting). All the material of this section comes from \cite{LTY}.\\[5pt]
Let us describe what is a \textit{semi-flat mirror pair} and how to construct it. Let $(X,\omega)$ be a symplectic manifold and $\pi: X \to B$ a Lagrangian torus bundle. Then there exist the \textit{Arnold-Liouville's action-angle coordinates}: for each $p \in B$ there exists an open neighborhood of $p \in U \subset B$ on which we have \textit{action coordinates} $r_1, \ldots, r_n$ of $U$ and a symplectomorphism $(\pi^{-1} (U), \omega) \cong (T^*U/\Lambda^*, \omega_{\mathrm{can}})$, where $\Lambda^*$ is a lattice bundle generated by $\del r_1, \ldots, \del r_n$. In this way the starting manifold is locally identified with a compact quotient of $T^*U$.
If we denote the corresponding fiber coordinates of $T^*U$ by  $\theta_i$, we can write $\omega = \sum_{i=1}^n d \theta_i \wedge d r_i$.\\[5pt]
We can now take the dual of our torus bundle, denoted by $\check{\pi}: \check{X} \to B$. $\check{X}$ is endowed with a canonical complex structure, constructed in the following way. Locally around each $p \in B$ there exists open neighborhood $p \in U \subset B$ and a biholomorphism $\check{\pi}^{-1} (U) \cong TU/\Lambda$, where $\Lambda \subset TU$ is the lattice bundle generated by $r_1, \ldots, r_n$ (the action coordinates mentioned before). Denoting the corresponding fiber coordinates of $TU$ by $\check{\theta}_i$, we can take on $\check{X}$ complex coordinates $\zeta_J := \exp (\check{\theta}_j + i r_j)$ or $z_j := \check{\theta}_j + i r_j$ for $j=1,\ldots,n$. In particular, the transition between local action coordinate systems on $B$ lies in $GL(n,\Z) \ltimes \R^n$.\\[5pt]
Suppose now that $\check{X}$ is endowed with a holomorphic volume form locally written as 
\[\check{\Omega} = d z_1 \wedge \ldots \wedge d z_n = (d\check{\theta}_1 + i d r_1) \wedge \ldots \wedge (d\check{\theta}_n + i d r_n)\]
(this is always the case when the action coordinate system can be chosen such that any of the already mentioned transitions lies in $SL(n,\Z) \ltimes \Z^n$).
\begin{definizione}
    Manifolds $(X,\omega)$ and $(\check{X}, \check{\Omega})$ as described are said to form a \textit{semi-flat mirror pair}.
\end{definizione}
The question which has to be addressed now is how the two manifolds $(X,\omega)$ and $(\check{X}, \check{\Omega})$ of a semi-flat mirror pair are related. The key for this is the \textit{Fourier-Mukai transform}, which we will now recall.
The Lagrangian torus bundle $X \rightarrow B$ induces via its fibers a \textit{polarization} $\Delta$ on the differential forms on $X$. More precisely, the right space to be considered is $\Omega^{p,q}_{B}(X;\C)$: it is the space of the one of differential forms depending only on the basis $B$ and with a \textit{formal} "bi-grading" such that $p$ runs on $\Delta$ and $q$ on $\Delta^{\perp}$ i.e., 
\[\alpha \in \Omega^{p,q}_{B}(X;\C) \iff \alpha = \sum_{\substack{p+q = k \\ i_1, ..., i_p \\ j_1, ..., j_q}}\alpha_{ij}(r)d\theta_{i_1} \wedge ... \wedge d\theta_{i_p} \wedge dr_{j_1} \wedge ... \wedge dr_{j_q}.\]
A similar space is considered on $\check{X}$: here we just have to take the \textit{actual} $(p,q)$-forms which again depend just on the basis $B$. We denote this space by $\Omega^{p,q}_{B}(\check{X};\C)$.
\begin{definizione}[\cite{LTY}, definition 4.3]
    The \textit{Fourier-Mukai transform} of a differential form $\check{\alpha} \in \Omega^{p,q}_{B}(\check{X};\C)$ is defined by 
    \begin{equation}\label{FMTrasnform}
        \text{FT}(\check{\alpha}) = \check{\pi}_*(\pi^*(\mathcal{P}\cdot \check{\alpha} \wedge e^{\sum d\check{\theta_i} \wedge d\theta_i}))
    \end{equation}
    where the maps $\pi, \check{\pi}$ are given by the fiber product
    \[\pi : (TB/\Lambda) \times (T^*B/\Lambda^*) \rightarrow TB/\Lambda\]
    \[\check{\pi} : (TB/\Lambda) \times (T^*B/\Lambda^*) \rightarrow T^*B/\Lambda^*\]
    and $\mathcal{P}$ is the \textit{polarization switch operator} \cite[definition 4.2]{LTY}.
\end{definizione}   
The next two Theorems from \cite{LTY} are the key to non-K\"ahler SYZ mirror symmetry. In fact, we can see that they connect in pairs non-K\"ahler Calabi-Yau manifolds with symplectic non-K\"ahler manifolds via the Fourier-Mukai transform, together with an isomorphism between the refined Bott-Chern and symplectic Bott-Chern cohomologies.
\begin{teorema}[\cite{LTY}, part of Theorem 5.1]
Let $\check{\omega} \in \Omega^{1,1}_{B}(\check{X};\C)$ be a real $(1,1)$-form on $\check{X}$. Denote by $\Omega$ the Fourier-Mukai transform of $e^{2 \check{\omega}}$. Then the following hold. 
\begin{enumerate}
\item $(\check{X},\check{\omega},\check{\Omega})$ forms an $SU(n)$ structure if and only if $(X, \omega, \Omega)$ forms an $SU(n)$ structure.  For such a pair, the conformal factor $F$ of $(X, \omega, \Omega)$ is related to the conformal factor $\check{F}$ of $(\check{X}, \check{\omega}, \check{\Omega})$ by $\check{F} = 2^{2n} F^{-1}.$
\item $(X, \omega, \Omega)$ is supersymmetric of Type-A if and only if $(\check{X}, \check{\omega}, \check{\Omega})$ is supersymmetric of Type-B.
\end{enumerate}
\end{teorema}
\begin{teorema}[\cite{LTY}, Theorem 4.5]
The Fourier-Mukai transform $\text{FT}$ gives an isomorphism
$$ \left(\Omega_B^* (X,\C), \frac{(-1)^ni}{2} d, \frac{(-1)^ni}{2} d^{\Lambda}\right) \cong \left(\Omega_B^* (\check{X},\C), \bar{\partial}, \partial\right) $$
as double complexes. In particular, $h^{p,q}_{B,TY}(X) = h^{n-p,q}_{B,BC}(\check{X})$ i.e., we have mirror Hodge diamonds for the refined cohomologies, which are defined by 
\begin{equation} \label{RefinedBCandTY}
H_{B,BC}^{p,q} (\check{X}) := \frac{\ker (d) \cap \Omega_B^{p,q}(\check{X})}{\text{im} (\partial\bar{\partial}) \cap \Omega_B^{p,q}(\check{X})} \qquad 
H^{(p,q)}_{B,d + d^\Lambda}(X) := \frac{\ker (d + d^\Lambda) \cap \Omega_B^{p,q}(X)}{\text{im} (dd^\Lambda) \cap \Omega_B^{p,q}(X)}.
\end{equation}
\end{teorema}
    
    \subsection{Compact quotients of Lie groups}
    We examine the two questions stated before in a precise setting, which is given by \textit{compact quotients of Lie groups by lattices}, where by \textit{lattice} we mean a discrete and co-compact subgroup.\\[5pt] 
    Precisely, we consider compact quotients of type $X=\Gamma \backslash G$, where $G$ is a connected and simply-connected Lie group and  $\Gamma$ is a lattice in $G$. We recall that in this setting $\pi_{1}(\Gamma \backslash G) \simeq \Gamma$. Two special cases of this construction are given by \textit{nilmanifolds} resp., \textit{solvmanifolds}: they correspond exactly to the group $G$ being a nilpotent resp., solvable Lie group. 
    \subsection{Setting and notations}
	We here fix the setting as well as some notations for this work. Let $ X = \Gamma \backslash G $ be a compact 6-dimensional manifold build as compact quotient of a connected and simply connected Lie group $G$ by a lattice $\Gamma$.\\[5pt]
    Suppose that $ X $ carries the structure of a compact non-K\"ahler Calabi-Yau manifold, with complex structure $ J $ given by a global co-frame of $ (1,0) $-forms $ \{\varphi^{1}, \varphi^{2}, \varphi^{3}\} $ such that the Calabi-Yau form $ \Omega $ is given by
	\begin{equation*}\label{FormaCalabiYau}
		\Omega = \varphi^{123}.
	\end{equation*}
	Such $ X $ carries a standard Hermitian metric $ g $ with associated $ (1,1) $-form
	\begin{equation*}\label{MetricaHermitiana}
		\omega = \dfrac{i}{2}(\varphi^{1\bar1} + \varphi^{2\bar2} + \varphi^{3\bar3})
	\end{equation*}
	which we require to be \textit{balanced} i.e., $ d\omega^2 = 0 $.\\[5pt]
	Setting $ \varphi^{i} = \theta^{i} + \sqrt{-1}\theta^{3+i} $, $ i= 1,2,3 $, we get a global co-frame of 1-forms $ \{\theta^{1}, ..., \theta^{6}\} $. Precisely, $ \theta^{1}, \theta^{2} $ and $ \theta^3 $ are the real parts of the $ (1,0) $ co-frame, while $ \theta^{4}, \theta^{5} $ and $ \theta^6 $ are the imaginary ones. The structure equations for this co-frame are, in their most generic form, given by
	\begin{equation*}\label{EqStrutturaGeneriche}
		d\theta^{i} = \sum_{1\le j < k \le 6}^{}f^{i}_{jk}\theta^j \wedge \theta^k \qquad f^{i}_{jk} \in C^{\infty}(X;\R).
	\end{equation*}
    We will occasionally make use of the following notation: $\theta^{ij} = \theta^{i} \wedge \theta^{j}$. Finally, we denote by $ \{E_1, ..., E_6\} $ the global frame for $ X $ dual to $ \{\theta^{1}, ..., \theta^{6}\} $.
	
	\section{Equations for SLags generated by invariant distributions}\label{Equations}
	Let $ X = \Gamma \backslash G $ be as in the previous section. Our strategy for finding \textit{SLags} is to find them as the foliation's 3-dimensional leaves coming from the integration of involutive distributions of triplets of vectors fields satisfying our SLags equations \eqref{DefSLag}. We point out that this method does not ensure compactness of the leaves i.e., that the SLags we find are compact. We address compactness in the examples of Section \ref{Examples}.\\[5pt]
	Consider an \textit{involutive distribution} $ \{V_1, V_2, V_3\} $ on $ X $ of the form
	\begin{equation}\label{GenericaDistribuzione}
		V_i = \sum_{j = 1}^{6} a_{ij}E_j \qquad a_{ij} \in C^{\infty}(X), i = 1,2,3.
	\end{equation}
	To such distribution we associate a $ 3 \times 6 $ matrix which collects all the coefficients $a_{ij}$:
	\begin{equation}\label{MatriceDistribuzione}
		A = \begin{pmatrix}
			a_{11} & a_{12} & a_{13} & a_{14} & a_{15} & a_{16} \\
			a_{21} & a_{22} & a_{23} & a_{24} & a_{25} & a_{26} \\
			a_{31} & a_{32} & a_{33} & a_{34} & a_{35} & a_{36} \\
		\end{pmatrix} =: \begin{pmatrix}
			x_1 & x_2 & x_3 & x_4 & x_5 & x_6 \\
			x_7 & x_8 & x_9 & x_{10} & x_{11} & x_{12} \\
			x_{13} & x_{14} & x_{15} & x_{16} & x_{17} & x_{18} \\
	\end{pmatrix}\\
	\end{equation}
	where we have denoted the coefficients with the variables $ x_i $ to write clearer equations in our main result. Combining equations \eqref{DefSLag} with the distributions of type \eqref{GenericaDistribuzione} we get the following lemma.
	\begin{lemma}\label{LemmaEqAlg}
		The distributions of type \eqref{GenericaDistribuzione} satisfying the \textit{SLags} equations \eqref{DefSLag} are the ones whose coefficients satisfy the system
		\begin{align}\label{EqAlgebricheSLags}
			\begin{cases}
				x_1x_{10} - x_4x_7 + x_2x_{11} - x_5x_8 + x_3x_{12} - x_6x_9 = 0 \\
				x_1x_{16} - x_4x_{13} + x_2x_{17} - x_5x_{14} + x_3x_{18} - x_6x_{15} = 0 \\
				x_7x_{16} - x_{10}x_{13} + x_8x_{17} - x_{11}x_{14} + x_9x_{18} - x_{12}x_{15} = 0 \\
				x_{1}x_{8}x_{18} - x_{1}x_{12}x_{14} - x_{2}x_{7}x_{18} +x_{2}x_{12}x_{13} + x_{6}x_{7}x_{14} - x_{6}x_{8}x_{13} + \\
				-x_{1}x_{9}x_{17} + x_{1}x_{11}x_{15} + x_{3}x_{7}x_{17} - x_{3}x_{11}x_{13} - x_{5}x_{7}x_{15} + x_{5}x_{9}x_{13} +  \\
				+x_{2}x_{9}x_{16}- x_{2}x_{10}x_{15} - x_{3}x_{8}x_{16} + x_{3}x_{10}x_{14} + +x_{4}x_{8}x_{15} - x_{4}x_{9}x_{14} + \\
				- x_{4}x_{11}x_{18} + x_{4}x_{12}x_{17} + x_{5}x_{10}x_{18} - x_{5}x_{12}x_{16} - x_{6}x_{10}x_{17} + x_{6}x_{11}x_{16} = 0.
			\end{cases}
		\end{align}
	\end{lemma}
	\begin{proof}
		We need to evaluate equations \eqref{DefSLag} on the vector fields $ V_1, V_2, V_3 $. The first equation $ \omega_{|_{L}} = 0 $ translates into the three equations
		\[ \omega(V_1, V_2) = 0 \qquad \omega(V_1, V_3) = 0 \qquad \omega(V_2, V_3) = 0. \]
		In real terms, $\omega$ is given by 
		\[ \omega = \theta^{14} + \theta^{25} + \theta^{36} \]
		therefore we just need to evaluate $ \theta^{14}, \theta^{25}  $ and $ \theta^{36} $ on the pairs $ (V_1, V_2), (V_1, V_3)$ and $ (V_2, V_3) $. Since we chose our frame and co-frame to be dual, it always holds (up to a multiplicative constant) that:
		\[ 	\theta^{ij}(V_k, V_h) = \theta^{i}(V_k)\theta^{j}(V_h) - \theta^{i}(V_h)\theta^{j}(V_k) = a_{ki}a_{hj} - a_{hi}a_{kj}. \]
		Therefore, we get:\\
		\begin{align*}
			 \theta^{14}(V_1, V_2) &= a_{11}a_{24} - a_{21}a_{14} \quad \theta^{25}(V_1, V_2) &= a_{12}a_{25} - a_{22}a_{15} \quad \theta^{36}(V_1, V_2) &= a_{13}a_{26} - a_{23}a_{16} \\
			 \theta^{14}(V_1, V_3) &= a_{11}a_{34} - a_{31}a_{14} \quad \theta^{25}(V_1, V_3) &= a_{12}a_{35} - a_{32}a_{15} \quad \theta^{36}(V_1, V_3) &= a_{13}a_{36} - a_{33}a_{16} \\
			 \theta^{14}(V_2, V_3) &= a_{21}a_{34} - a_{31}a_{24} \quad \theta^{25}(V_2, V_3) &= a_{22}a_{35} - a_{32}a_{25} \quad \theta^{36}(V_2, V_3) &= a_{23}a_{36} - a_{33}a_{26} \\
		\end{align*}
		which substituting the $ a_{ij} $ with the $ x_k $ from the matrix equality \eqref{MatriceDistribuzione} and summing each line gives the first three equations of the system \eqref{EqAlgebricheSLags}.\\[5pt]
		For the second equation $ \im\Omega_{|_{L}} = 0 $ we firstly need to express $ \Omega $ in terms of our real co-frame. This gives:
		\[ \im\Omega =  \theta^{126} - \theta^{456} + \theta^{234} - \theta^{135} \]
		therefore we need to evaluate these four terms on the triplet $ (V_1, V_2, V_3) $. As before we have (up to a multiplicative constant):
		\[ \theta^{ijk}(V_1, V_2, V_3) = a_{1i}a_{2j}a_{3k} - a_{1i}a_{2k}a_{3j} - a_{1j}a_{2i}a_{3k} - a_{1k}a_{2j}a_{3i} + a_{1k}a_{2i}a_{3j} + a_{1j}a_{2k}a_{3i}. \]
		This then gives:
		\begin{align*}
			\theta^{126}(V_1, V_2, V_3) &= a_{11}a_{22}a_{36} - a_{11}a_{26}a_{32} - a_{12}a_{21}a_{36} - a_{16}a_{22}a_{31} + a_{16}a_{21}a_{32} + a_{12}a_{26}a_{31} \\
			\theta^{456}(V_1, V_2, V_3) &= a_{14}a_{25}a_{36} - a_{14}a_{26}a_{35} - a_{15}a_{24}a_{36} - a_{16}a_{25}a_{34} + a_{16}a_{24}a_{35} + a_{15}a_{26}a_{34} \\
			\theta^{234}(V_1, V_2, V_3) &= a_{12}a_{23}a_{34} - a_{12}a_{24}a_{33} - a_{13}a_{22}a_{34} - a_{14}a_{23}a_{32} + a_{14}a_{22}a_{33} + a_{13}a_{24}a_{32} \\
			\theta^{135}(V_1, V_2, V_3) &= a_{11}a_{23}a_{35} - a_{11}a_{25}a_{33} - a_{13}a_{21}a_{35} - a_{15}a_{23}a_{31} + a_{15}a_{21}a_{33} + a_{13}a_{25}a_{31} \\
		\end{align*}
		which again substituting from matrix \eqref{MatriceDistribuzione} gives the fourth equation of \eqref{EqAlgebricheSLags} and concludes the proof.
	\end{proof}
	\begin{osservazione}
		Lemma \ref{LemmaEqAlg} provides \textit{purely algebraic} equations to identify possible integrable distributions which can then generate \textit{SLags}. We remark that here we do not take into account involutivity of the distributions. In fact, this is a fact depending on the behaviour of vector fields on the manifold $ X $ and therefore cannot be checked in a generic setting. However, the advantage of equations \eqref{EqAlgebricheSLags} lies in the fact that they can be easily evaluated and therefore applied to a wide range of examples.\\[2pt]
		We wish to point out another aspect: if we have an involutive distribution satisfying \eqref{EqAlgebricheSLags}, we can not tell in general if the \textit{SLag} submanifolds to which the distribution integrates are compact. As involutivity, this property depends on the group $ G $ and the lattice $ \Gamma $. In Section \ref{Examples}, we show some compactness explicitly for some SLags.
	\end{osservazione}

	\subsection{A special case: distributions of type $ \{E_i, E_j, E_k\} $}\label{SpecialDistr}
	The simplest examples of distributions $ \{V_1, V_2, V_3\} $ as \eqref{GenericaDistribuzione} are the ones where just one coefficient $ a_{ij} $ for vector field appears. These distributions include all the ones of type $ \{E_i, E_j, E_k\} $, for $ i \ne j \ne k $, which translates in $ a_{1i} = a_{2j} = a_{3k} = 1 $ and the rest of the matrix $ A $ being zero. These are the ones we consider from now on.\\[5pt]
    It is easy to check which of these distributions satisfy the systems of Lemma \ref{LemmaEqAlg}. We find the following.
	\begin{proposizione}\label{DistribuzioniSingole}
		Let $ X $ be as fixed in Section \ref{Preliminaries}. Then the only distributions of type $ \{E_i, E_j, E_k\} $ which satisfy the SLags equations \eqref{DefSLag} are given by
		\begin{align*}
			\{E_1, E_2, E_3\} \qquad \{E_1, E_5, E_6\} \qquad \{E_2, E_4, E_6\} \qquad \{E_3, E_4, E_5\} 
		\end{align*}
	\end{proposizione}
	\begin{proof}
		Straightforward by evaluation of equations \eqref{EqAlgebricheSLags}.
	\end{proof}
	\begin{osservazione}
		As already pointed out for equations \eqref{EqAlgebricheSLags}, Proposition \ref{DistribuzioniSingole} does not guarantee that the four distributions actually integrate to compact \textit{SLags}. In fact, we need to check integrability (i.e., the distributions need to be involutive) and, in case integability holds, compactness.
	\end{osservazione}
    We now want to discuss a sort of "symmetric" setting. In the proof of Lemma \ref{LemmaEqAlg} we compute our equations and then find the distributions of Proposition \ref{SpecialDistr} by explicitly writing $\omega$ and $\Omega$. These distributions integrate to angle zero SLags. However, if we consider angle $-\frac{\pi}{2}$ SLags i.e., the Calabi-Yau form $\Omega$ gets multiplied by $i$, then by the same reasoning we can find four distributions as the ones of Proposition \ref{SpecialDistr}, which are in a sort of way "complementary" to them. Precisely, the following holds.
    \begin{proposizione}\label{SpecialDistrRotated}
        Let $X$ be as in Section \ref{Preliminaries} but with Calabi-Yau form $\tilde{\Omega} = i\Omega$ (or equivalently considering SLags of angle $-\frac{\pi}{2}$). Then the only distributions of type $ \{E_i, E_j, E_k\} $ which satisfy the SLags equations \eqref{DefSLag} are given by
		\begin{align*}
			\{E_4, E_5, E_6\} \qquad \{E_2, E_3, E_4\} \qquad \{E_1, E_3, E_5\} \qquad \{E_1, E_2, E_6\} 
		\end{align*}
    \end{proposizione}
    It is straightforward to see that the theory of deformations and the examples we next develop can be considered either for the distributions of Proposition \ref{SpecialDistr} or \ref{SpecialDistrRotated} and as a matter of fact we consider the ones in the first proposition.
    However, the distributions of Proposition \ref{SpecialDistrRotated} are crucial for considering fibrations over compact manifolds, as we will need for the constructions on sections \ref{MirrorSymmCSNakamura} and \ref{MirrorSymmCPNakamura}.
	\section{Equations for deformations of special type SLags}\label{Deformations}
	In the previous section we focused on existence of special Lagrangian submanifolds. We now turn to first order deformations, whose background we described in Section \ref{Preliminaries}. Suppose that a distribution $ \{E_i, E_j, E_k\} $ from our Proposition \ref{DistribuzioniSingole} integrates to a SLag. We carry out explicit computations of equations \eqref{EquazioniDef}. In particular, the equations we find agree with the computations for the Iwasawa manifold carried out in \cite[Section 3]{CGPS}. We prove the following.
	\begin{teorema}\label{TeoremaEquazioniDeformazioni}
		Let $ X $ be a non-K\"ahler Calabi-Yau manifold as fixed in Section \ref{Preliminaries}. Let $ \{E_i, E_j, E_k\} $ be a distribution among the four of Proposition \ref{DistribuzioniSingole} which we suppose to be involutive and integrating to a submanifold (a SLag) that we denote by $ L_{ijk} $. Let $ T^*L_{ijk} = \langle \theta^{i}, \theta^j, \theta^k \rangle $. Then the space of infinitesimal deformations of $ L_{ijk} $ is described by the 1-forms $ \alpha = \alpha_i\theta^i + \alpha_j\theta^j + \alpha_k\theta^k \in T^*L_{ijk}$ whose coefficients satisfy the following system of differential equations on $ L_{ijk} $:
		\begin{align}\label{SistemaPDEDeformazioni}
				&E_i(\alpha_i) + \alpha_if_{jk} + E_j(\alpha_j) - \alpha_jf_{ik} + E_k(\alpha_k) + \alpha_kf_{ij} = 0\\
                &E_i(\alpha_j)-E_j(\alpha_i) + \alpha_i[(1-(-1)^{F(i,\tilde{i})})f^{i}_{ij} + (-1)^{F(i,\tilde{i})}f^{\tilde{i}}_{j\tilde{i}} - (-1)^{F(j,\tilde{j})}f^{\tilde{j}}_{i\tilde{i}}] + \\
                &+\alpha_j[(1-(-1)^{F(j,\tilde{j})})f^{j}_{ij}- (-1)^{F(j,\tilde{j})}f^{\tilde{j}}_{i\tilde{j}} + (-1)^{F(i,\tilde{i})}f^{\tilde{i}}_{j\tilde{j}}] + \nonumber\\
                &+\alpha_k[(1-(-1)^{F(k,\tilde{k})})f^{k}_{ij}+(-1)^{F(i,\tilde{i})}f^{\tilde{i}}_{j\tilde{k}} - (-1)^{F(j,\tilde{j})}f^{\tilde{j}}_{i\tilde{k}}] = 0\nonumber\\
                &E_i(\alpha_k)-E_k(\alpha_i) + \alpha_i[(1-(-1)^{F(i,\tilde{i})})f^{i}_{ik} + (-1)^{F(i,\tilde{i})}f^{\tilde{i}}_{k\tilde{i}} - (-1)^{F(k,\tilde{k})}f^{\tilde{k}}_{i\tilde{i}}] +\\
                &+\alpha_j[(1-(-1)^{F(j,\tilde{j})})f^{j}_{ik}- (-1)^{F(k,\tilde{k})}f^{\tilde{k}}_{i\tilde{j}} + (-1)^{F(i,\tilde{i})}f^{\tilde{i}}_{k\tilde{j}}] + \nonumber\\
                &+\alpha_k[(1-(-1)^{F(k,\tilde{k})})f^{k}_{ik}+(-1)^{F(i,\tilde{i})}f^{\tilde{i}}_{k\tilde{k}} - (-1)^{F(k,\tilde{k})}f^{\tilde{k}}_{i\tilde{k}}] = 0\nonumber\\
                &E_j(\alpha_k)-E_k(\alpha_j) +\alpha_i[(1-(-1)^{F(i,\tilde{i})})f^{i}_{jk} + (-1)^{F(j,\tilde{j})}f^{\tilde{j}}_{k\tilde{i}} - (-1)^{F(k,\tilde{k})}f^{\tilde{k}}_{j\tilde{i}}] +\\
            &+\alpha_j[(1-(-1)^{F(j,\tilde{j})})f^{j}_{jk}+ (-1)^{F(j,\tilde{j})}f^{\tilde{j}}_{k\tilde{j}} - (-1)^{F(k,\tilde{k})}f^{\tilde{k}}_{j\tilde{j}}] + \nonumber\\
            &+\alpha_k[(1-(-1)^{F(k,\tilde{k})})f^{k}_{jk}+(-1)^{F(j,\tilde{j})}f^{\tilde{j}}_{k\tilde{k}} - (-1)^{F(k,\tilde{k})}f^{\tilde{k}}_{j\tilde{k}}] = 0  \nonumber             
        \end{align}
		where $ d\theta^{rs} = f_{rs}\theta^{ijk}  $ for $r,s \in \{i,j,k\}$, $ f^{t}_{rs}$ the same as in the structure equations of $ \{\theta^1, ..., \theta^6\} $ but \underline{restricted to $L$}, an index of type $ \tilde{h} $ denotes the one among $ \{1, ..., 6\} $ such that $ JE_h = E_{\tilde{h}}$ and $ F(i,j) $ is defined as
		\[ F(i, j) = \begin{cases}
			0 & i \le j \\
			1 & i > j
		\end{cases}. \]
	\end{teorema}
	\begin{osservazione}
		\begin{enumerate}
			\item Precisely, in the previous statement the forms $ \theta^{\bullet} $ are to be considered restricted to $ L_{ijk} $. However, we keep denoting them by $\theta^{\bullet}  $.
			\item We have defined $ F(i,i) = 0 $ for completness however, we will not need this in the proof and/or in the statement as we will find objects of type $ F(i, \tilde{i}) $ and an index $ h $ and its corresponding $ \tilde{h} $ will always be different.
		\end{enumerate}
	\end{osservazione}
	\begin{proof}(Theorem \ref{TeoremaEquazioniDeformazioni}) Let $ \alpha \in T^*L_{ijk} $ written as in the statement. We use the convention $ i < j <k $, which is not restrictive since it just corresponds to a permutation of $ \{E_i, E_j, E_k\} $. We fix the metric $ g $ associated to $\omega$, which is written in this co-frame as
		\[ g = \theta^i \otimes \theta^i + \theta^j \otimes \theta^j + \theta^k \otimes \theta^k\]
		and $ \{\theta^i, \theta^j, \theta^k\} $ is $ g $-orthonormal basis.\\[5pt]
		Starting from the second equation of \eqref{EquazioniDefomazione}, which translates in $ d(*\alpha )= 0 $, we compute:
		\begin{align*}
			0 &= d(*\alpha) = d(\alpha_i\theta^{jk} - \alpha_j\theta^{ik} + \alpha_k\theta^{ij}) = \\
			&= [E_i(\alpha_i) + \alpha_if_{jk} + E_j(\alpha_j) - \alpha_jf_{ik} + E_k(\alpha_k) + \alpha_kf_{ij}]\theta^{ijk}
		\end{align*}
		which gives the first equation of \eqref{SistemaPDEDeformazioni}. For the next equation, we start by computing $ d\alpha $:
		\begin{align*}
			d\alpha &= d(\alpha_i\theta^i + \alpha_j\theta^j + \alpha_k\theta^k) = \\
			& = -E_j(\alpha_i)\theta^{ij} - E_k(\alpha_i)\theta^{ik} + \alpha_i(f^{i}_{ij}\theta^{ij} + f^{i}_{ik}\theta^{ik} + f^{i}_{jk}\theta^{jk}) + \\
			& + E_i(\alpha_j)\theta^{ij} - E_k(\alpha_j)\theta^{jk} + \alpha_j(f^{j}_{ij}\theta^{ij} + f^{j}_{ik}\theta^{ik} + f^{j}_{jk}\theta^{jk}) + \\
			& + E_i(\alpha_k)\theta^{ik} + E_j(\alpha_k)\theta^{jk} + \alpha_k(f^{k}_{ij}\theta^{ij} + f^{k}_{ik}\theta^{ik} + f^{k}_{jk}\theta^{jk}) = \\
			&= [E_i(\alpha_j)-E_j(\alpha_i) + \alpha_if^{i}_{ij} + \alpha_jf^{j}_{ij} +  \alpha_kf^{k}_{ij}]\theta^{ij} + \\
			&+ [E_i(\alpha_k)- E_k(\alpha_i) + \alpha_if^{i}_{ik} + \alpha_jf^{j}_{ik} + \alpha_k f^{k}_{ik}]\theta^{ik} + \\
			&+ [E_j(\alpha_k)- E_k(\alpha_j) + \alpha_if^{i}_{jk} + \alpha_jf^{j}_{jk} + \alpha_k f^{k}_{jk}]\theta^{jk}.
		\end{align*}
		Finally, we need to compute $ T\alpha = -d\omega(J\alpha^{*}, ., .) $, where $ \alpha^* $ is the $ g $-dual vector field to $\alpha$. We get:
		\[ \alpha^* = \alpha_iE_i + \alpha_jE_j + \alpha_kE_k \qquad J\alpha^{*} = \alpha_iE_{\tilde{i}} + \alpha_jE_{\tilde{j}} + \alpha_kE_{\tilde{k}}. \]
		To compute $ d\omega $ in a useful form for our next computations, we observe that we can write
		\[ \omega = (-1)^{F(i,\tilde{i})}\theta^{i\tilde{i}} +  (-1)^{F(j,\tilde{j})}\theta^{j\tilde{j}} +  (-1)^{F(k,\tilde{k})}\theta^{k\tilde{k}} \]
    	(we need the functions $F$ as we are always writing $\omega$ with the lowest index before and the highest after).
		In its most general form i.e., applying structure equations \eqref{EqStrutturaGeneriche}, $ d\omega $ becomes
		\begin{align*}
			d\omega &= (-1)^{F(i,\tilde{i})}\bigg[\sum_{1\le r< s \le 6}f^{i}_{rs}\theta^{rs\tilde{i}} - \sum_{1\le  r< s \le 6}f^{\tilde{i}}_{rs}\theta^{irs}\bigg] + \\
			&+ (-1)^{F(j,\tilde{j})}\bigg[\sum_{1\le r< s \le 6}f^{j}_{rs}\theta^{rs\tilde{j}} - \sum_{1\le r< s \le 6}f^{\tilde{j}}_{rs}\theta^{jrs}\bigg] + \\
			&+ (-1)^{F(k,\tilde{k})}\bigg[\sum_{1\le r< s \le 6}f^{k}_{rs}\theta^{rs\tilde{k}} - \sum_{1\le r< s \le 6}f^{\tilde{k}}_{rs}\theta^{krs}\bigg]. 
		\end{align*}
		We remark that several terms in the expression of $ d\omega $ above are actually zero because of repeating indices. At the same time, several other will cancel out after the next evaluation due to the fact that $ \theta^{\tilde{i}}_{|_{L_{ijk}}} = \theta^{\tilde{j}}_{|_{L_{ijk}}} = \theta^{\tilde{k}}_{|_{L_{ijk}}} = 0$. We then get:
		\begin{align*}
			&d\omega(J\alpha^*, ., .) = \alpha_id\omega(E_{\tilde{i}}, ., .) + \alpha_jd\omega(E_{\tilde{j}}, ., .) + \alpha_kd\omega(E_{\tilde{k}}, . ,.) = \\
			&= \alpha_i((-1)^{F(i,\tilde{i})}[f^{i}_{ij}\theta^{ij} + f^{i}_{ik}\theta^{ik} + f^{i}_{jk}\theta^{jk} - f^{\tilde{i}}_{j\tilde{i}}\theta^{ij} - f^{\tilde{i}}_{k\tilde{i}}\theta^{ik}] +\\
            &+(-1)^{F(j,\tilde{j})}[f^{\tilde{j}}_{i\tilde{i}}\theta^{ij} - f^{\tilde{j}}_{k\tilde{i}}\theta^{jk}] + (-1)^{F(k,\tilde{k})}[f^{\tilde{k}}_{i\tilde{i}}\theta^{ik} + f^{\tilde{k}}_{j\tilde{i}}\theta^{jk}]) + \\
			& +\alpha_j((-1)^{F(j,\tilde{j})}[f^{j}_{ij}\theta^{ij} + f^{j}_{ik}\theta^{ik} + f^{j}_{jk}\theta^{jk} + f^{\tilde{j}}_{i\tilde{j}}\theta^{ij} - f^{\tilde{j}}_{k\tilde{k}}\theta^{jk}] +\\
            &-(-1)^{F(i,\tilde{i})}[f^{\tilde{i}}_{j\tilde{j}}\theta^{ij} - f^{\tilde{i}}_{k\tilde{k}}\theta^{ik} ]+ (-1)^{F(k,\tilde{k})}[f^{\tilde{k}}_{i\tilde{j}}\theta^{ik} + f^{\tilde{k}}_{j\tilde{j}}\theta^{jk}]) +\\
			& +\alpha_k((-1)^{F(k,\tilde{k})}[f^{k}_{ij}\theta^{ij} + f^{k}_{ik}\theta^{ik} + f^{k}_{jk}\theta^{jk} + f^{\tilde{k}}_{i\tilde{k}}\theta^{ik} + f^{\tilde{k}}_{j\tilde{k}}\theta^{jk}] +\\
            &-(-1)^{F(i,\tilde{i})}[f^{\tilde{i}}_{j\tilde{k}}\theta^{ij} - f^{\tilde{i}}_{k\tilde{k}}\theta^{ik}] + (-1)^{F(j,\tilde{j})}[f^{\tilde{j}}_{i\tilde{k}}\theta^{ij} - f^{\tilde{j}}_{k\tilde{k}}\theta^{jk}]).
		\end{align*}
        We finally compute the equation $d\alpha + T\alpha = 0$. We group together the three possible contributions: $\theta^{ij}$, $\theta^{ik}$ and $\theta^{jk}$. Listing the coefficients:
        \begin{align*}
            \theta^{ij}:\quad &E_i(\alpha_j)-E_j(\alpha_i) +\\
            &+ \alpha_i[(1-(-1)^{F(i,\tilde{i})})f^{i}_{ij} + (-1)^{F(i,\tilde{i})}f^{\tilde{i}}_{j\tilde{i}} - (-1)^{F(j,\tilde{j})}f^{\tilde{j}}_{i\tilde{i}}] +\\
            &+\alpha_j[(1-(-1)^{F(j,\tilde{j})})f^{j}_{ij}- (-1)^{F(j,\tilde{j})}f^{\tilde{j}}_{i\tilde{j}} + (-1)^{F(i,\tilde{i})}f^{\tilde{i}}_{j\tilde{j}}] + \\
            &+\alpha_k[(1-(-1)^{F(k,\tilde{k})})f^{k}_{ij}+(-1)^{F(i,\tilde{i})}f^{\tilde{i}}_{j\tilde{k}} - (-1)^{F(j,\tilde{j})}f^{\tilde{j}}_{i\tilde{k}}]
        \end{align*}
        \begin{align*}
            \theta^{ik}:\quad &E_i(\alpha_k)-E_k(\alpha_i) +\\
            &+ \alpha_i[(1-(-1)^{F(i,\tilde{i})})f^{i}_{ik} + (-1)^{F(i,\tilde{i})}f^{\tilde{i}}_{k\tilde{i}} - (-1)^{F(k,\tilde{k})}f^{\tilde{k}}_{i\tilde{i}}] +\\
            &+\alpha_j[(1-(-1)^{F(j,\tilde{j})})f^{j}_{ik}- (-1)^{F(k,\tilde{k})}f^{\tilde{k}}_{i\tilde{j}} + (-1)^{F(i,\tilde{i})}f^{\tilde{i}}_{k\tilde{j}}] + \\
            &+\alpha_k[(1-(-1)^{F(k,\tilde{k})})f^{k}_{ik}+(-1)^{F(i,\tilde{i})}f^{\tilde{i}}_{k\tilde{k}} - (-1)^{F(k,\tilde{k})}f^{\tilde{k}}_{i\tilde{k}}]
        \end{align*}
        \begin{align*}
            \theta^{jk}:\quad &E_j(\alpha_k)-E_k(\alpha_j) +\\
            &+ \alpha_i[(1-(-1)^{F(i,\tilde{i})})f^{i}_{jk} + (-1)^{F(j,\tilde{j})}f^{\tilde{j}}_{k\tilde{i}} - (-1)^{F(k,\tilde{k})}f^{\tilde{k}}_{j\tilde{i}}] +\\
            &+\alpha_j[(1-(-1)^{F(j,\tilde{j})})f^{j}_{jk}+ (-1)^{F(j,\tilde{j})}f^{\tilde{j}}_{k\tilde{j}} - (-1)^{F(k,\tilde{k})}f^{\tilde{k}}_{j\tilde{j}}] + \\
            &+\alpha_k[(1-(-1)^{F(k,\tilde{k})})f^{k}_{jk}+(-1)^{F(j,\tilde{j})}f^{\tilde{j}}_{k\tilde{k}} - (-1)^{F(k,\tilde{k})}f^{\tilde{k}}_{i\tilde{k}}]
        \end{align*}
		Therefore, we finally obtain for the second deformation equation-
		which gives the remaining three equations of system \eqref{SistemaPDEDeformazioni}.
	\end{proof}

	\begin{osservazione}
		The solution of system \eqref{SistemaPDEDeformazioni} varies depending on the manifold we work on. However, the form of the system is general and completely accords with the explicit computations performed for example in \cite[Section 3]{CGPS} and the ones we present in the next section (which we had also computed separately before without Theorem \ref{TeoremaEquazioniDeformazioni}).
	\end{osservazione}
	
	\section{Explicit computations on nilmanifolds and solvmanifolds}\label{Examples}
	We here apply the results of the previous section, in particular Proposition \ref{DistribuzioniSingole}, in the setting of \textit{nilmanifolds} and \textit{solvmanifolds}. What is special about these examples in the possibility to prove that all the distributions of Proposition \ref{DistribuzioniSingole} integrate to \textit{SLags} and to actually solve the deformation equations \eqref{SistemaPDEDeformazioni}. In particular, we get a peculiar behaviour on the completely solvable Nakamura manifold, showing that the four foliations of \textit{SLags} determined by Proposition \ref{DistribuzioniSingole} are not in general diffeomorphic. More precisely, this shows that topologiccaly different special Lagrangian submanifolds can co-exist inside the same manifold.
	
	\subsection{The Iwasawa manifold $ \mathbb{I}_3 $}
	We apply our results to the \textit{Iwasawa manifold} $ \mathbb{I}_3 $. In particualr, we generalize the examples of \cite[Section 3]{CGPS}. Consider the Iwasawa manifold $ \mathbb{I}_3 =   \Gamma \backslash \mathbb{H}(3;\C)$ where $ \Gamma = \mathbb{Z}^3 \oplus \sqrt{-1}\mathbb{Z}^3 $. We fix on $ \mathbb{H}(3;\C) \simeq \C^3 $ coordinates $ (z_1, z_2, z_3) $ (complex) or $ (x_1, x_2, x_3, y_1, y_2, y_3) $ (real), with $ z_j = x_j + y_j$. A complex structure $ J $ on $ \mathbb{I}_3 $ is fixed through the global $ (1,0) $-co-frame $ \{\phi^1, \phi^2, \phi^3\} $, defined by $ \phi^{1} = \theta^{1} + i\theta^{4} $, $ \phi^{2} = \theta^{2} + i\theta^{5}  $, $ \phi^{3} = \theta^{3} + i\theta^{6}  $ using the real global co-frame $ \{\theta^1, \theta^2, \theta^3, \theta^4, \theta^5, \theta^6\} $, given by
		\begin{align*}
			\theta^1 = dx_1 \qquad  \theta^2 = dx_2 \qquad \theta^3 = dx_3 - x_1dx_2 + y_1dy_2\\
			\theta^4 = dy_1 \qquad \theta^5 = dy_2 \qquad \theta^6 = dy_3 - y_1dx_2 - x_1dy_2
		\end{align*}
		with structure equations
		\begin{align}\label{EqStrIwasawa}
			d\theta^1 = d\theta^2 = d\theta^4 = d\theta^5 = 0 \qquad d\theta^{3} = -\theta^{12} + \theta^{45} \qquad d\theta^{6} = \theta^{24} - \theta^{15}.
		\end{align}
		The corresponding real dual frame, denoted by $ \{E_1, E_2, E_3, E_4, E_5, E_6 \} $, is given by
		\begin{align*}
			E_1 = \dfrac{\del}{\del x_1} \qquad E_2 = \dfrac{\del}{\del x_2} + x_1\dfrac{\del}{\del x_3} + y_1\dfrac{\del}{\del y_3} \qquad E_3 = \dfrac{\del}{\del x_3} \\
			E_4 = \dfrac{\del}{\del y_1} \qquad E_5 =  \dfrac{\del}{\del y_2} + x_1\dfrac{\del}{\del y_3} - y_1\dfrac{\del}{\del x_3} \qquad E_6 = \dfrac{\del}{\del y_3}
		\end{align*}
		with brackets
		\begin{equation}\label{bracketsIwasawa}
			[E_1, E_2] = E_3 \qquad [E_1, E_5] = E_6 \qquad [E_2, E_4] = -E_6 \qquad [E_4, E_5] = -E_3.
		\end{equation}
		The standard hermitian metric on $ \mathbb{I}_{3} $ is fixed via the $ (1,1) $ non-degenerate (but not symplectic!) form
		 \[  \omega =\dfrac{i}{2} (\theta^{14} + \theta^{25} + \theta^{36}).   \]
		 Recalling the Iwasawa manifold is not K\"ahler, we get a non-K\"ahler Calabi-Yau structure from the holomorphic $ (3,0) $-form 
		 \[   \Omega = \phi^{123}\]
		 for which we have
		 \[\text{Re}\Omega = \theta^{123} - \theta^{156} + \theta^{246} - \theta^{345} \qquad \im\Omega =  \theta^{126} - \theta^{135} + \theta^{234} - \theta^{456} . \]
		We now look at \textit{Slags} following our Proposition \ref{DistribuzioniSingole}. In this case, looking at the brackets \eqref{bracketsIwasawa} all distributions given by the proposition are involutive and therfore integrable. We explicitely write the integral submanifolds corresponding to the first foru distributions (SLags of angle zero):
			\begin{align*}
			\mathcal{L}_{123} &= \{[x_1, x_2, x_3, A, B, C+Ax_2] \in \mathbb{I}_3: A, B, C \in \mathbb{R}\} \\
			\mathcal{L}_{156} &= \{[x_1, A, C - By_2, B, y_2, y_3] \in \mathbb{I}_3: A, B, C \in \mathbb{R}\} \\
			\mathcal{L}_{246} &= \{[x_1, A, B + Ax_1, y_1, C, y_3] \in \mathbb{I}_3: A, B, C \in \mathbb{R}\} \\
			\mathcal{L}_{345} &= \{[x_1, A, x_3, y_1, B, C + Bx_1] \in \mathbb{I}_3: A, B, C \in \mathbb{R}\}.
		\end{align*}
	\begin{osservazione}
		\begin{enumerate}
			\item As the computations in \cite[Section 3]{CGPS} show, the SLag in $ \mathcal{L}_{123} $ corresponding to $ A=B=C=0 $ is compact.
			\item In this case, we have $ b_{1}(L_{ijk}) = 2 $ for all $ i,j,k $ and for all submanifolds in the families. Therefore, these \textit{SLags} share similar topological information.
		\end{enumerate}
	\end{osservazione}
    We now generalize the computations of \cite[Section 3]{CGPS} by applying Theorem \ref{TeoremaEquazioniDeformazioni} to a SLag $ L_{123} \in \mathcal{L}_{123}$. The distribution corresponding to $ L_{123} $ is $ \{E_1, E_2, E_3\} $ (as always, they are restricted to $ L_{123} $ but keep the same notation) and the dual frame is just $ \{\theta^1, \theta^2, \theta^3\} $. In particular, from the structure equations \eqref{EqStrIwasawa} we get
	\[ d\theta^{12} = d\theta^{13} = d\theta^{23} = 0  \]
	and therefore the first of the deformation equations is just
	\[ E_1(\alpha_1) + E_2(\alpha_2) + E_3(\alpha_3) = 0.\]
	For the remaining three equations, again from \eqref{EqStrIwasawa} we get that all $ f_{rs}^{t} $ vanish except for
	\[ f_{12}^{3} = f_{15}^{6} = -1 \qquad f_{45}^{3} = f_{24}^{6} = 1. \]
	Since here $ \{i,j,k\} = \{1,2,3\} $ and $ \{\tilde{i}, \tilde{j}, \tilde{k}\} = \{4,5,6\} $, grouping together the four equations of Theorem \ref{TeoremaEquazioniDeformazioni} we get the final system
	\begin{equation*}
		\begin{cases}
			E_1(\alpha_1) + E_2(\alpha_2) + E_3(\alpha_3) = 0 \\
			E_1(\alpha_2) - E_2(\alpha_1) = 0 \\
			E_1(\alpha_3) - E_3(\alpha_1) +\alpha_2 = 0 \\
			E_2(\alpha_3) - E_3(\alpha_2) -\alpha_1 = 0 \\
		\end{cases}
	\end{equation*}
	which perfectly corresponds to \cite[Equation 3.10]{CGPS}, except for the fact that here we have $ u = 0 $. The system can be completely solved via elliptic theory just using the bracket $ [E_1, E_2] = E_3 $ and the resulting space of deformations is 1-dimensional, less that one with respect to $ b_1(L_{123}) = 2 $. \\
    We point out that the Theorem can be applied to all the SLags we found however, we leave things as they are for $\mathbb{I}_3$ and focus on a less studied example, which is the \textit{Nakamura manifold}.
	\subsection{The completely solvable 6-dimensional Nakamura manifold}
	We consider the \textit{completely solvable Nakamura manifold} $ N = \Gamma \ltimes_{\rho}( \C \ltimes_{\rho} \C^{2}) $, as in \cite{CT} and \cite{LT}. Let us briefly recall the construction. Let $M \in \SL(2,\mathbb{Z})$ be a matrix with positive eigenvalues $\{e^{\lambda_{1}}, e^{\lambda_{2}}\}$ such that
\begin{equation}
	PMP^{-1} = \text{diag}(e^{\lambda_{1}}, e^{\lambda_{2}}) =: D
\end{equation}
with $P \in \GL(n,\mathbb{R})$. Since $M \in \SL(n, \mathbb{Z})$, we have $\sum_{i = 1}^2 \lambda_i = 0$ .
Let $\rho : \mathbb{C} \rightarrow GL(2, \mathbb{C})$ be the group action of $\mathbb{C}$ over $\mathbb{C}^{2}$ given by:
\[ w \longmapsto \text{diag}(e^{\frac{1}{2}\lambda_1(w + \overline{w})}, e^{\frac{1}{2}\lambda_2(w + \overline{w})}).\]
We can then consider the (real) Lie group $ G := \mathbb{C} \ltimes_{\rho} \mathbb{C}^{2}$ where the product, denoted by $*$, is explicitly given by
\[ (w', z'_{1}, z'_{2}) * (w, z_{1}, z_{2}) = (w' + w, \rho(w')(z_{1}, z_{2}) + (z'_{1},z'_{2})) \]
As was proved in \cite[Subsection 4.5]{CT}, $G$ is a 2-step solvable not nilpotent Lie group of \textit{completely solvable type}. We obtain compact quotients of $(G, *)$ by fixing $\tau \in \mathbb{R} \setminus \{0\}$ and considering the lattice in $G$ (for the proof see \cite[Section 3]{CT}) given by $\Gamma_{P, \tau} :=  \Gamma'_{\tau} \ltimes_{\rho} \Gamma''_{P}$, where
\begin{equation}\label{LatticesInNakamura}
    \Gamma'_{\tau} := \mathbb{Z} \oplus \tau\sqrt{-1}\cdot\mathbb{Z} \qquad
    \Gamma''_{P} := P\mathbb{Z}^{2} \oplus \sqrt{-1}P\mathbb{Z}^{2}
\end{equation}
Then $N:=\Gamma_{P, \tau} \backslash (\mathbb{C} \ltimes_{\rho} \mathbb{C}^{2})$ is a solvmanifold, called the \textit{(generalized) Nakamura manifold} associated to the triple $(M, P, \tau)$.
A complex structure on $ N $ is fixed by the global $ (1,0) $-co-frame
	\[ \varphi^0 = dw \qquad \varphi^1 = e^{-\frac{1}{2}\lambda_1(w+\bar w)}dz_1 \qquad \varphi^2 = e^{-\frac{1}{2}\lambda_2(w+\bar w)}dz_2.\]
	Re-writing in real terms following the notation of Section \ref{Preliminaries} i.e., $ \varphi^0 = \theta^1 + i\theta^4 $, $ \varphi^1 = \theta^2 + i\theta^5 $, $ \varphi^2 = \theta^3 + i\theta^6 $ gives the global co-frame
	\[ \theta^1 = dx \qquad \theta^4 = dy\]
	\[ \theta^2 = e^{-\frac{1}{2}\lambda_1(w+\bar w)}dx_1  \qquad \theta^5  = e^{-\frac{1}{2}\lambda_1(w+\bar w)}dy_1 \]
	\[ \theta^3 = e^{-\frac{1}{2}\lambda_2(w+\bar w)}dx_2 \qquad \theta^6 = e^{-\frac{1}{2}\lambda_2(w+\bar w)}dy_2\]
	with real structure equations:
	\[ d\theta^{1} = d\theta^4 = 0 \]
	\[ d\theta^2 = -\lambda_1 \theta^{12} \qquad d\theta^5 = -\lambda_1 \theta^{15}\]
	\[ d\theta^3 = -\lambda_2 \theta^{13} \qquad d\theta^6 = -\lambda_2 \theta^{16}. \\ \]
	A global frame $ \{E_1, E_2, E_3, f_0, f_1, f_2\} $ dual to $ \{\theta^{1}, ..., \theta^6\} $ is given by
	\begin{align*}
		E_1 = \dfrac{\del}{\del x} \qquad E_2 = e^{\frac{1}{2}\lambda_1(w+\bar w)}\dfrac{\del}{\del x_1} \qquad E_3= e^{\frac{1}{2}\lambda_2(w+\bar w)}\dfrac{\del}{\del x_2} \\
		E_4 = \dfrac{\del}{\del y} \qquad E_5 = e^{\frac{1}{2}\lambda_1(w+\bar w)}\dfrac{\del}{\del y_1} \qquad E_6 = e^{\frac{1}{2}\lambda_2(w+\bar w)}\dfrac{\del}{\del y_2} \\
	\end{align*}
	with non-zero brackets
	\[ [E_1, E_2] = \lambda_1E_2 \qquad [E_1, E_3] = \lambda_2E_3 \qquad [E_1, E_5] = \lambda_1E_5 \qquad [E_1, E_6] = \lambda_2E_6.\]
	In \cite[Proposition 4.4]{CT}, it is proved that $ N $ cannot be K\"ahler. We give $ N $ a non-K\"ahler Calabi-Yau structure by the holomorphic non-zero $ (3,0) $-form
	\[ \Omega = \varphi^{012}\]
	and the hermitian metric defined by the $ (1,1) $-form 
	\[ \omega = \theta^{14} + \theta^{25} + \theta^{36}. \]
	We can then apply our Proposition \ref{DistribuzioniSingole} and the four distributions are all involutive therefore integrable and as for $ \mathbb{I}_3 $ they integrate to four foliations of $ N $ whose leaves are all \textit{SLags}. We get the following foliations of $N$:
    \begin{align*}
			\mathcal{L}_{123} &= \{[x, x_1, x_2, A, B, C] \in N: A, B, C \in \mathbb{R}\} \\
			\mathcal{L}_{156} &= \{[x, A, B, C, y_1, y_2] \in N: A, B, C \in \mathbb{R}\} \\
			\mathcal{L}_{246} &= \{[A, x_1, B, y, C, y_2] \in N: A, B, C \in \mathbb{R}\} \\
			\mathcal{L}_{345} &= \{[A, B, x_2, y, y_1, C ] \in N: A, B, C \in \mathbb{R}\}.
	\end{align*}
    If instead we consider the SLags of angle $\theta = -\frac{\pi}{2}$, we get instead the foliations
    \begin{align*}
			\mathcal{L}_{456} &= \{[A, B, C, y, y_1, y_2] \in N: A, B, C \in \mathbb{R}\} \\
			\mathcal{L}_{234} &= \{[A, x_1, x_2, y, B, C] \in N: A, B, C \in \mathbb{R}\} \\
			\mathcal{L}_{135} &= \{[x, A, x_2, B, y_1, C] \in N: A, B, C \in \mathbb{R}\} \\
			\mathcal{L}_{126} &= \{[x, x_1, A, B, C, y_2] \in N: A, B, C \in \mathbb{R}\}.
	\end{align*}	
    \subsubsection{Topological properties of the SLags}
    We give a detailed study of the topology of the previous SLags foliations. The main goal is to show how here different foliations share substantial topological differences. In the following results of the section, we will refer to the \textit{completely solvable Nakamura manifold} and as the \textit{Nakamura manifold}.\\[5pt]
    To begin with, the structure equations (in particular, the presence of $\theta^1$ in all the non-vanishing differentials) directly show the following.
	\begin{teorema}\label{NumBettiSLagsCSNakamura}
		Let $ N $ be the 6-dimensional Nakamura manifold. Then all the previous distributions \ref{DistribuzioniSingole} integrate to SLags which behave topologically in the following way:
		\begin{itemize}
			\item the submanifolds in $ \mathcal{L}_{123} $, $ \mathcal{L}_{156} $, $\mathcal{L}_{135}$ and $\mathcal{L}_{126}$ have $ b_{1}(L_{ijk}) = 1 $;
			\item the submanifolds in $ \mathcal{L}_{246} $, $ \mathcal{L}_{345} $, $\mathcal{L}_{234}$ and $\mathcal{L}_{456}$ have $ b_{1}(L_{ijk}) = 3$.
		\end{itemize}
	\end{teorema} 
	Therefore, we obtain topologically distinct families of special Lagrangians in completely solvable  Nakamura manifold.
	\begin{corollario}\label{EsistenzaSLagTopDiverse}
		There can exists topologically distinct families of special Lagrangians in completely solvable  Nakamura 6-manifold. 
	\end{corollario}
    Theorem \ref{NumBettiSLagsCSNakamura} groups the foliations in two different families. We can study topological properties of the foliations, namely \textit{closure} and \textit{compactness} of their leaves, more into detail.\\[10pt]
    Recall that our previous results of section \ref{Equations} do not ensure the compactness of any leaf in general. In fact. a leaf $L$ is compact if and only if there exists a rank 3 sub-lattice preserving its universal cover. This yields more stringent requirements on the constants $A, B, C$ and yields a variety of topologically different SLag foliations. We summarize all this in the following result.
   \begin{teorema}\label{TopologyFoliations}
   Let $N$ be the 6-dimensional Nakamura manifold as before. Then the following hold.
   \begin{enumerate}
       \item A leaf of the foliation $\mathcal{L}_{123}$ is a closed (and therefore compact) submanifold of $N$ if and only if $B = C = 0$. Similarly, a leaf of $\mathcal{L}_{156}$ is closed if and only if $A = B = 0$.
       \item No leaf of the foliations $\mathcal{L}_{246}$, $\mathcal{L}_{345}$, $\mathcal{L}_{135}$ and $\mathcal{L}_{126}$ is a closed submanifold of $N$.
       \item All the leaves of the foliations $\mathcal{L}_{234}$, $\mathcal{L}_{456}$ are closed submanifolds of $N$.
    \end{enumerate}
   \end{teorema}
The proof of Theorem \ref{TopologyFoliations} is contained in a series of intermediate lemmata and propositions. We start addressing case (1).
\begin{proposizione}\label{Topology123And156}
     A leaf of the foliation $\mathcal{L}_{123}$ is a closed (and therefore compact) submanifold of $N$ if and only if $B = C = 0$. Similarly, a leaf of $\mathcal{L}_{156}$ is closed if and only if $A = B = 0$.
\end{proposizione}
\begin{proof}
Let us consider $\mathcal{L}_{123}$.
Given $X = (x, x_1, x_2, A,B,C) \in \R^6$ and $(a, a_1, a_2, b, b_1, b_2) \in \mathbb{Z}^6$, the image under the
action of the lattice is
\[\tilde{X} = (x + a, e^{\lambda a}x_1 + p_{11}a_{1} + p_{12}a_2,e^{-\lambda a}x_2 + p_{21}a_{1} + p_{22}a_2, A + b\tau, e^{\lambda a}B + p_{11}b_{1} + p_{12}b_2, e^{-\lambda a}C + p_{21}b_{1} + p_{22}b_2)\]
and we need to determine a rank 3 sub-lattice so that
\[\tilde{X} \in \{y = A, y_1 = B, y_2 = C\}.\]
Of course we must have $b = 0$. Instead, the other conditions read
\[e^{\lambda a}B + p_{11}b_{1} + p_{12}b_2 = B \qquad e^{-\lambda a}C + p_{21}b_{1} + p_{22}b_2 = C.\]
These equations should be satisfied by all the elements of any rank 1 sub-lattice, which implies in particular that for every $n \in \mathbb{N}$ 
\[e^{\lambda na}B + np_{11}b_{1} + np_{12}b_2 = B \qquad e^{-\lambda na}C + np_{21}b_{1} + np_{22}b_2 = C.\]
This can hold if and only if we have $B = C = 0$ and therefore $b_1 = b_2 = 0$. Therefore, the rank 3 sub-lattice can exist if and only if $B = C = 0$ and it is given by points $(a, a_1, a_2) \in \mathbb{Z}$ acting on $X$ as
\[X \longmapsto (x + a, e^{\lambda a}x_1 + p_{11}a_{1} + p_{12}a_2,e^{-\lambda a}x_2 + p_{21}a_{1} + p_{22}a_2, A, 0, 0).\]
The proof for $\mathcal{L}_{156}$ is carried out in the same way, working on the second and third coordinates (which leads $A = B = 0$).
\end{proof}
Following the same reasoning we can address cases (2) and (3). Case (3) is easier.
\begin{proposizione}
    All the leaves of the foliations $\mathcal{L}_{234}$, $\mathcal{L}_{456}$ are closed submanifolds of $N$.
\end{proposizione}
\begin{proof}
    Proceeding as before, we just need to observe that the equations we must impose in this case read ($\mathcal{L}_{234}$ on the left, $\mathcal{L}_{456}$ on the right)
    \[\begin{cases}
        A + a = 0 \\
        e^{\lambda a}B + p_{11}b_{1} + p_{12}b_2 = B \\
        e^{-\lambda a}C + p_{21}b_{1} + p_{22}b_2 = C \\
    \end{cases} \text{ or }\begin{cases}
        A + a = 0 \\
        e^{\lambda a}B + p_{11}a_{1} + p_{12}a_2 = B \\
        e^{-\lambda a}C + p_{21}a_{1} + p_{22}a_2 = C \\
    \end{cases}. \]
    In both cases first equation leads to $a = 0$, which reduces the remaining two equations to
    \[P\binom{b_1}{b_2} = \binom{0}{0} \text{ or } P\binom{a_1}{a_2} = \binom{0}{0}\]
    giving $b_1 = b_2 = 0$ resp., $a_1 = a_2 = 0$. Therefore, the rank 3 sub-lattice always exists and every leaf in closed (therefore compact).
\end{proof}
Finally, we solve case (2): this requires a sort of number-theoretic argument. We start from $\mathcal{L}_{246}$ and $\mathcal{L}_{345}$.
\begin{lemma}\label{Lemma246}
    A leaf of the special Lagrangian foliations $\mathcal{L}_{246}$ or $\mathcal{L}_{345}$ is a closed submanifold of $N$ if and only if there exists a matrix $M'$ with integer entries, and both non-zero columns such that 
    \[PM' = D\]
    for some diagonal matrix $D$.
\end{lemma}
\begin{proof}
    Let us consider $\mathcal{L}_{246}$; the proof for $\mathcal{L}_{345}$ is the same just by inter-changing $(a_1,a_2)$ with $(b_1, b_2)$. Working as in the proof of Proposition \ref{Topology123And156}, the conditions for having the rank 3 sub-lattice preserved translate into
    \[\begin{cases}
        a = 0 \\
        p_{11}b_{1} + p_{12}b_{2} = 0\\
        p_{21}a_1 + p_{22}a_2 = 0.
    \end{cases}\]
    In other words, we need an integer valued matrix 
    \[M' = \begin{pmatrix}
        a_1 & b_1 \\
        a_2 & b_2 \\
    \end{pmatrix}\]
    with both non-vanishing columns such that $PM' = D$ is a diagonal matrix.
\end{proof}
We shall then ask whether such a matrix exists and the following argument proves the answer is negative.
\begin{lemma}\label{LemmaNumTh}
    Let $M \in \SL(2;\mathbb{Z})$, $M \ne \pm I$, such that there exists a matrix $P \in \GL(2;\R)$ for which $PMP^{-1} = D = \text{diag}(d_1, d_2)$ diagonal matrix. Then there exist no integer-valued matrix $M' = \begin{pmatrix}
        a_1 & b_1 \\
        a_2 & b_2 \\
    \end{pmatrix}$ with both non-vanishing columns such that $PM'$ is a diagonal matrix.
\end{lemma}
\begin{proof}
    Consider the equation $PM' = D'$ diagonal matrix. It translates into the equations
    \[p_{11}b_1 + p_{12}b_2 = 0 \qquad p_{21}a_1 + p_{22}a_2=0. \]
    Notice that each one of these equations, let us take the first one, is satisfied if and only if there exists a non-zero constant $c \in \R^*$ such that $(cp_{11}, cp_{12}) \in \mathbb{Z}^2 $ (it will also be non-zero being $P$ invertible). We show that this cannot happen.\\[5pt]
    Consider the characteristic polynomial of $M$, which is given by
    \[\chi(M)(\lambda) = \lambda^2 - \text{Tr}(M)\lambda + 1\]
    and whose roots are 
    \[\lambda_{\pm} = \dfrac{\text{Tr}(M) \pm \sqrt{\text{Tr}(M)^2 - 4}}{2}.\]
    We first show that $\text{Tr}(M)^2 - 4$ is not a perfect square and therefore $\lambda_\pm$ are both irrational. Indeed, suppose $\text{Tr}(M)^2 - 4 = k^2 \in \mathbb{Z}$. Then $\lambda_\pm \in \frac{1}{2}\mathbb{Z}$ since $M$ is a integer matrix. Then two things can happen:
    \begin{itemize}
        \item $\lambda_\pm \in \mathbb{Z}$: since $\det(M)= 1$ $\lambda_\pm$ can only be both 1 or both -1, which leads to $M = \pm I$;
        \item $\lambda_\pm \in \frac{1}{2}\mathbb{Z}$ (i.e., at least one of them is not a multiple of 2): then the only two possibilities are $(\lambda_+,\lambda_-) = (\pm 2, \pm\frac{1}{2})$, but then $\text{Tr}(M) = \lambda_+ + \lambda- \ne \mathbb{Z}$.
    \end{itemize}
    Therefore, the eigenvalues of $M$ are irrational.\\[5pt]
    Let $v_+ = \binom{x_+}{y_+}$ be an eigenvector of eigenvalue $\lambda_+$. Since $Mv_+ = \lambda_+v_+$ we get equations
    \[ (m_{11} - \lambda_+)x_+ + m_{12}y_+ = 0 \qquad m_{21}x_+ + (m_{22} - \lambda_+)y_+ = 0 \]
    with $m_{11} - \lambda_+, m_{22} - \lambda_+$ irrational, while $ m_{12}, m_{21} \in \mathbb{Z}$. This implies in particular that there cannot exist any constant $c \in \R^*$ such that $c(x_+, y_+) \in \mathbb{Z}^2$ (otherwise by simple multiplication and substitutions we would get $m_{11} - \lambda_+, m_{22} - \lambda_+$ rational). Same argument with an eigenvector $v_- = \binom{x_-}{y_-}$ of eigenvalue $\lambda_-$ show that there cannot exist any $d \in \R^*$ such that $d(x_-, y_-) \in \mathbb{Z}^2$.\\[5pt]
    Finally, consider $M^{T}$ instead of $M$, for which it still holds $M^{T} \in \SL(2; \mathbb{Z}) \setminus \{\pm I\}$ and $(P^T)^{-1}M^TP^T = D$. Recall in particular that the columns of $P^T$ are eigenvectors of $M^T$. Therefore, applying the previous argument as before we can show that there do not exist constants $c, d \in \R^*$ such that $c(p_{11}, p_{12}), d(p_{21}, p_{22}) \in \mathbb{Z}^2$ which proves the lemma.
\end{proof}
From Lemmata \ref{Lemma246}, \ref{LemmaNumTh} we finally get the following.
\begin{proposizione}\label{Topology246and345}
    No leaf of the special Lagrangian foliations $\mathcal{L}_{246}$, $\mathcal{L}_{345}$ is a closed submanifold of $N$.
\end{proposizione}
Now that we have the previous lemmata, the case of  $\mathcal{L}_{135}$ and $\mathcal{L}_{126}$ is  easier.
\begin{proposizione}
    No leaf of the special Lagrangian foliations $\mathcal{L}_{135}$ and $\mathcal{L}_{126}$ is a closed submanifold of $N$.
\end{proposizione}
\begin{proof}
    Consider $\mathcal{L}_{135}$ (the reasoning for $\mathcal{L}_{126}$ is the same) and work as in the proof of Proposition \ref{Topology123And156}. If we want a rank 3 sub-lattice to be preserved, we here get again $b = 0$ and 
    \[e^{-\lambda a}A + p_{11}a_1 + p_{12}a_2 = A \qquad e^{\lambda a}C + p_{21}b_1 + p_{22}b_2 = C.\]
    Again, this must hold for any rank 1 sub-lattice therefore for every $n \in \mathbb{N}$ it must be
    \[e^{-\lambda na}A + np_{11}a_1 + np_{12}a_2 = A \qquad e^{\lambda na}C + np_{21}b_1 + np_{22}b_2 = C\]
    which leads $A = C = 0$ and equations
    \[p_{11}a_1 + p_{12}a_2 = 0 \qquad p_{21}b_1 + p_{22}b_2 = 0\]
    which by Lemma \ref{LemmaNumTh} (applied interchanging the columns of $M'$) are never simultaneously satisfied for both non-vanishing $(a_1, a_2)$, $(b_1, b_2)$.
\end{proof}
\vspace{0.5cm}
It emerges from Theorem \ref{TopologyFoliations} that the foliations $\mathcal{L}_{234} $ and $\mathcal{L}_{456}$ are peculiar among all the other. We can prove that in this case the leaves are actually real 3-tori.
   \begin{proposizione}\label{ToriLeaves}
       Every leaf of the special Lagrangian foliations $\mathcal{L}_{234} $ and $\mathcal{L}_{456}$ is diffeomorphic to a real 3-torus.
   \end{proposizione}
   \begin{proof}
       We carry out the proof for a generic leaf $L_{234} \in \mathcal{L}_{234}$. By the construction of the Nakamura manifold (see for example \cite{CT}) we fix matrices $M \in \SL(2;\mathbb{Z}), P \in \GL(2;\R)$ such that
       \[PMP^{-1} = \begin{pmatrix}
           e^{\lambda} & 0 \\
           0 & e^{-\lambda}\\
       \end{pmatrix}\]
       for $\lambda \in \R, \lambda \ne 0$. Now consider any point $p = [A:x_1:x_2:y:B:C] \in L_{234}$. Recall in particular that $A, B$ and $C$ are fixed. Define the matrix $\tilde{P} \in \GL(2;\R)$ by
       \[\tilde{P} = \begin{pmatrix}
           e^{A\lambda} & 0 \\
           0 & e^{-A\lambda}\\
       \end{pmatrix}P.\]
       We directly compute
       \begin{align*}
        \tilde{P}M\tilde{P^{-1}}&=\begin{pmatrix}
           e^{A\lambda} & 0 \\
           0 & e^{-A\lambda}\\
       \end{pmatrix}PMP^{-1}\begin{pmatrix}
           e^{-A\lambda} & 0 \\
           0 & e^{A\lambda}\\
       \end{pmatrix} = \\
       &=\begin{pmatrix}
           e^{A\lambda} & 0 \\
           0 & e^{-A\lambda}\\
       \end{pmatrix}\begin{pmatrix}
           e^{\lambda} & 0 \\
           0 & e^{-\lambda}\end{pmatrix}\begin{pmatrix}
           e^{-A\lambda} & 0 \\
           0 & e^{A\lambda}\\
       \end{pmatrix} = \begin{pmatrix}
           e^{\lambda} & 0 \\
           0 & e^{-\lambda}\end{pmatrix}
           \end{align*}
    therefore also for every $a \in \mathbb{Z}$ it holds
    \begin{equation}\label{MatrixTorus}
        \tilde{P}(M^a)\tilde{P}^{-1} = \begin{pmatrix}
        e^{\lambda a} & 0 \\
        0 & e^{-\lambda a}\\
    \end{pmatrix}
    \end{equation}
    with $M^a \in \SL(2;\mathbb{Z})$. Now define a lattice in $\R^3$ by $\Gamma := \tilde{P}\mathbb{Z} \oplus\tau\mathbb{Z}$, so that the quotient $\Gamma \backslash \R^3$ is a real 3-torus. We can now define a map 
    \begin{align*}
        \varphi: &L_{234} \longrightarrow \Gamma \backslash \R^3 \qquad q \longmapsto [e^{-\lambda A}x_1:e^{\lambda A}x_2:y]
    \end{align*}
    which if well-defined gives the diffeomorphism that we are looking for. Precisely, $\varphi$ is defined on every element of the equivalence class $q$ by the projection on the second, third and fourth components; we show that this is well-defined on the quotient in $\R^3$. Every element in the equivalence class of $q$ is of type
    \[(A + a, e^{\lambda a}x_1 + p_{11}a_{1} + p_{12}a_2,e^{-\lambda a}x_2 + p_{21}a_{1} + p_{22}a_2, y + b\tau, e^{\lambda a}B + p_{11}b_{1} + p_{12}b_2, e^{-\lambda a}C + p_{21}b_{1} + p_{22}b_2)\]
    and such a point gets sent to 
    \[(e^{-\lambda(A+a)}[e^{\lambda a}x_1 + p_{11}a_{1} + p_{12}a_2],e^{\lambda(A+a)}e[^{-\lambda a}x_2 + p_{21}a_{1} + p_{22}a_2], y + b\tau)\]
    which we re-write as
    \begin{equation}\label{EquivPoint}
    (e^{-\lambda A}x_1 + e^{-\lambda a}(e^{-\lambda A}p_{11}a_{1} + e^{-\lambda A}p_{12}a_2),(e^{\lambda A}x_2 + e^{\lambda a}(e^{\lambda A}p_{21}a_{1} + e^{\lambda A}p_{22}a_2), y + b\tau).
    \end{equation}
    Since \eqref{MatrixTorus} holds we can now see that
    \begin{align*}
        \begin{pmatrix}
            e^{-\lambda a} & 0 \\
            0 & e^{\lambda a}
        \end{pmatrix}\tilde{P}\begin{pmatrix}
            a_1 \\
            a_2\\
        \end{pmatrix} = \begin{pmatrix}
            e^{-\lambda a} & 0 \\
            0 & e^{\lambda a}
        \end{pmatrix}\begin{pmatrix}
            e^{\lambda a} & 0 \\
            0 & e^{-\lambda a}
        \end{pmatrix}\tilde{P}(M^a)^{-1}\begin{pmatrix}
            a_1 \\
            a_2\\
        \end{pmatrix} = \tilde{P}\begin{pmatrix}
            \tilde{a_1} \\
            \tilde{a_2}\\
        \end{pmatrix}
    \end{align*}
    for some $\tilde{a_1}, \tilde{a_2} \in \mathbb{Z}$. This means that point \eqref{EquivPoint} can be written as 
    \[(e^{-\lambda A}x_1 + \tilde{p_{11}}\tilde{a_1} + \tilde{p_{12}}\tilde{a_2},e^{\lambda A}x_2 + \tilde{p_{21}}\tilde{a_1} + \tilde{p_{22}}\tilde{a_2}, y + b\tau)\]
    and it is then equivalent to $(e^{-\lambda A}x_1, e^{\lambda A}x_2, y)$ in $\Gamma \backslash \R^3$.
   \end{proof}
 \subsubsection{Deformation theory}
	Now we study deformation theory, applying Theorem \ref{TeoremaEquazioniDeformazioni} to compute the dimension of the deformation space for the compact SLags in the foliations $\mathcal{L}_{123}$ and $\mathcal{L}_{156}$.\\[5pt]
    \underline{$ L_{123} \in \mathcal{L}_{123} $}. We can here proceed as in $\mathbb{I}_3$, with $ \{i,j,k\} = \{1,2,3\} $ and $ \{\tilde{i}, \tilde{j}, \tilde{k}\} = \{4,5,6\} $. We find the system:
	\begin{equation*}
		\begin{cases}
			E_1(\alpha_1) + E_2(\alpha_2) + E_3(\alpha_3) = 0 \\
			E_1(\alpha_2) - E_2(\alpha_1) +\lambda\alpha_2= 0 \\
			E_1(\alpha_3) - E_3(\alpha_1) - \lambda\alpha_3= 0 \\
			E_2(\alpha_3) - E_3(\alpha_2)  = 0 \\
		\end{cases}
	\end{equation*}
	where we have set $ \lambda_1 = \lambda = -\lambda_2 $. We start by applying $ E_1  $ to the first equation, $ E_2 $ to the second equation and $ E_3 $ to the third one, getting
	\begin{align*}
		E_{1}^{2}(\alpha_1) + E_1E_2(\alpha_2) + E_1E_3(\alpha_3) &= 0 \\
		E_2E_1(\alpha_2) - E_2^{2}(\alpha_1) + \lambda E_2(\alpha_2) &= 0\\
		E_3E_1(\alpha_3) - E_3^{2}(\alpha_1) - \lambda E_3(\alpha_3) &= 0.
	\end{align*}
	Subtracting the last two from the first one we get to
	\[ E_{1}^{2}(\alpha_1) + E_{2}^{2}(\alpha_1) + E_{3}^{2}(\alpha_1) = 0 \]
	which by ellipticity and compactness of $ N $ implies that $\alpha_1$ \textit{must be constant}. Therefore, the system reduces to
	\begin{equation*}
		\begin{cases}
			E_2(\alpha_2) + E_3(\alpha_3) = 0 \\
			E_1(\alpha_2) +\lambda \alpha_2= 0 \\
			E_1(\alpha_3) - \lambda\alpha_3= 0 \\
			E_2(\alpha_3) - E_3(\alpha_2)  = 0. \\
		\end{cases}
	\end{equation*}
	Now we apply $ E_2 $ and $ E_3 $ both to the first and last equations:
	\begin{equation*}
		\begin{cases}
			E_2^{2}(\alpha_2) + E_2E_3(\alpha_3) = 0 \\
			E_3E_2(\alpha_2) + E_3^{2}(\alpha_3) = 0 \\
			E_3E_2(\alpha_3) - E_3^{2}(\alpha_2)  = 0 \\
			E_2^{2}(\alpha_3) - E_2E_3(\alpha_2)  = 0 \\
		\end{cases}
	\end{equation*}
	and suitably subtracting them we get the two equations
	\[ E_2^{2}(\alpha_2) + E_3^{2}(\alpha_2) = 0 \qquad E_2^{2}(\alpha_3) + E_3^{2}(\alpha_3) = 0\]
	which imply for the same argument as before that $ \alpha_2 = \alpha_2(x) $ and $\alpha_3 = \alpha_3(x)$. So now we can directly find $ \alpha_2 $ and $ \alpha_3 $ from the second and third equations of the second system we found, which simply gives
	\[ \alpha_2 = c_2e^{-\lambda x} \qquad \alpha_3 = c_3e^{\lambda x} \]
	for some constants $ c_2, c_3 $. However, we know that all the functions $\alpha_i$ must be $ \mathbb{Z} $-periodic in $ x $ for the construction of $ N $: since $ \lambda \ne 0, 2k\pi $ ($\lambda$ is an eigenvalue of a matrix $M \in \SL(2,\mathbb{Z)}$) the only possibility is to have $ \alpha_2 = \alpha_3 = 0 $. Therefore, the deformation space is again 1-dimensional, but in this case it behaves \textit{accordingly} to $ b_1(L_{123}) $.\\[5pt]
	\underline{$L_{156} \in \mathcal{L}_{156}$}: applying Theorem \ref{TeoremaEquazioniDeformazioni}, in this case we get the system
    \[\begin{cases}
        E_1(\alpha_1) + E_5(\alpha_5) + E_6(\alpha_6) &= 0\\
        E_1(\alpha_5) - E_5(\alpha_1) -3\lambda\alpha_5 &= 0\\
        E_1(\alpha_6) - E_6(\alpha_1) +3\lambda\alpha_6 &= 0\\
        E_5(\alpha_6) - E_6(\alpha_5) &= 0\\
    \end{cases}\]
    which has to be addressed in a slightly different manner. The brackets of the distribution are similar to the previous case:
    \[[E_1, E_5] = \lambda E_5 \qquad [E_1, E_6] = -\lambda E_6 \qquad [E_5, E_6] = 0.\]
    We start by applying $E_5$ to the first equation, $E_1$ to the second and $E_6$ to the fourth, getting
    \[\begin{cases}
        E_5E_1(\alpha_1) + E_5^2(\alpha_5) + E_5E_6(\alpha_6) &= 0\\
        E_1^2(\alpha_5) - E_1E_5(\alpha_1) -3\lambda E_1(\alpha_5) &= 0\\
        E_6E_5(\alpha_6) - E_6^2(\alpha_5) &= 0.\\
    \end{cases}\]
    Adding the first two equations and subtracting the third one, we get 
    \begin{align*}
        0 &= E_1^2(\alpha_5) + E_5^2(\alpha_5) + E_6^2(\alpha_5) -\lambda E_5(\alpha_1) -3\lambda E_1(\alpha_5) =\\
        &= E_1^2(\alpha_5) + E_5^2(\alpha_5) + E_6^2(\alpha_5) -4\lambda E_1(\alpha_5) -3\lambda^2\alpha_5 
    \end{align*}
    which is an elliptic equation for $\alpha_5$. Using compactness, it must be $\alpha_5 = c_5 = \text{const.}$. The same reasoning for $\alpha_6$ yields $\alpha_6 = c_6 = \text{const.}$. Therefore, the initial system reduces to
    \[\begin{cases}
        E_1(\alpha_1) &= 0\\
        E_5(\alpha_1) +3\lambda c_5 &= 0\\
        E_6(\alpha_1) -3\lambda c_6 &= 0.\\
    \end{cases}\]
    Applying $E_1$ to the first equation, $E_5$ to the second and $E_6$ to the third one we get the elliptic equation 
    \[(E_1^2 + E_5^2 + E_6^2)(\alpha_1) = 0\]
    which gives us $\alpha_1 = \text{const.}$. Finally, being $\alpha_1$ constant the last system shows that it must be $\alpha_5 = \alpha_6 = 0$. Therefore, the space of first order deformations is again 1-dimensional and behaves \textit{accordingly} to $b_1(L_{156})$.
\begin{osservazione}
        The dimension of the space of first order deformations perfectly agrees with the fact that, in each case, we have explicitly constructed a 1-dimensional family of compact SLags inside $\mathcal{L}_{123}$ and $\mathcal{L}_{156}$.  Therefore, the deformation theory is unobstructed.
    \end{osservazione}

\section{Non-K\"ahler mirror symmetry: The completely solvable Nakamura manifold}\label{MirrorSymmCSNakamura}
Consider the 6-dimensional Nakamura manifold $N$ as in the previous section. We use the foliation $\mathcal{L}_{234}$ to apply the ideas of Strominger-Yau-Zaslow \cite{SYZ} mirror symmetry and actually compute a mirror for $N$, following the proposal of Lau-Tseng-Yau \cite{LTY}. Recall that by Proposition \ref{ToriLeaves} every leaf of the special Lagrangian foliation $\mathcal{L}_{234}$ is diffeomorphic to a real 3-torus. Let us start with the following lemma.
\begin{lemma}
    Let $\sigma \in \mathcal{L}_{156}$ be the closed special Lagrangian given by $A = B = C = 0$ (see Proposition \ref{Topology123And156}). If $L \in \mathcal{L}_{234}$ is any special Lagrangian 3-torus then $L \cap \sigma$ is a point. That is, $\sigma$ is a section of the special Lagrangian foliation.
\end{lemma}
\begin{proof}
    Let $L = L_{ABC} \in \mathcal{L}_{234} $. Then a point $p$ is in $L \cap \sigma$ if and only if 
    \[p = [A, x_1,x_2,B,C] = [x, 0, 0, 0, y_1, y_2]\]
and there is a unique solution of this equation.
\end{proof}
This means that we get a special Lagrangian torus fibration over $N$
\[T \hookrightarrow N \rightarrow B\]
where $B = \sigma \in \mathcal{L}_{156}$ and the tori fiber are the tori of $\mathcal{L}_{234}$, which we recall to be diffeomorphic to
\[ T :=\R^3 / \Lambda \qquad \text{ where } \Lambda = P\mathbb{Z}^2 \oplus \tau\mathbb{Z} \]
(note that here on the fibers over $\sigma$ we have $A = 0$ so that the matrix $\tilde{P}$ in the proof of Proposition \ref{ToriLeaves} coincides with $P$).\\[5pt]
In order to now apply the idea of SYZ mirror symmetry we observe that the dual tori $T^\vee$ of $T$ are given by 
\[ T^\vee :=\R^3 / \Lambda^\vee \qquad \Lambda^\vee = (P^{-1})^T\mathbb{Z}^2 \oplus \tau^{-1}\mathbb{Z}. \]
We now use all this information to compute the mirror of the Nakamura manifold. Before proceeding, we check that $(N, \omega, \Omega)$ satisfy the type IIB equations \cite[equations (2.1) to (2.4)]{LTY}. First two equations of the system have already been checked. For the third, we have
\[\omega^3 = -6\theta^{123456} \qquad \Omega \wedge \bar\Omega = 8i\theta^{123456}\]
so that $F = 8$. Finally, we get 
\[ 2i\del\delbar\omega =4\lambda^2(\theta^{1245} - \theta^{1346}) \implies 2i\del\delbar(\frac{1}{8}\omega) = \frac{1}{2}\lambda^2(\theta^{1245} - \theta^{1346}) =: \rho_B.\]
Let us start by writing the fibration $ N \rightarrow B$ more explicitly. In the coordinates, the map is given by
\[(x, x_1, x_2, y, y_2, y_1) \longmapsto (x, y_1, y_2):= (r_0, r_1, r_2)\]
and (passing of course to the quotient) the fibers of the map are diffeomorphic to the real 3-tori $T = \R^3 / \Lambda$ previously defined. We address the construction of the mirror in several steps. 
\subsection{Identification with $TB/\Lambda$}
We start by identifying $N$ holomorphically with $TB/\Lambda_{TB}$, where $\Lambda_{TB}$ is a lattice bundle which we will determine later. We know that $B = \Gamma_B \backslash \R^3$, where $\Gamma_B$ is the lattice $\mathbb{Z} \oplus P\mathbb{Z}^2$ acting via
\[(r_0, r_1, r_2) \longmapsto (r_0 + a, e^{\lambda a}r_1 + p_{11}b_1 + p_{12}b_2, e^{-\lambda a}r_2 + p_{21}b_1 + p_{22}b_2). \]
This is exactly the affine transformation 
\[ \begin{pmatrix}
    r_0 \\
    r_1 \\
    r_2 
\end{pmatrix} = \begin{pmatrix}
    1 & 0 & 0 \\
    0 & e^{\lambda a} & 0\\
    0 & 0 & e^{-\lambda a}
\end{pmatrix}\begin{pmatrix}
    r_0 \\
    r_1 \\
    r_2 
\end{pmatrix} + \begin{pmatrix}
    a \\
    P\begin{pmatrix}
    b_1 \\
    b_2 \\
\end{pmatrix}
\end{pmatrix} \]
therefore $B$ is endowed with a natural affine structure, with affine local coordinates $(r_0, r_1, r_2)$. We denote by $(\alpha_0, \alpha_1, \alpha_2)$ the fiber coordinates on $TB$, namely 
\[(\alpha_0, \alpha_1, \alpha_2) \longmapsto \alpha_0\dfrac{\del}{\del r_0} + \alpha_1\dfrac{\del}{\del r_1} +\alpha_2\dfrac{\del}{\del r_2}.\]
Therefore, the tangent space to $TB$ is spanned by the vector fields 
\[\bigg\{\dfrac{\del}{\del r_i}, \dfrac{\del}{\del \alpha_i}\bigg\}_{i = 1,2,3}\]
and there is a natural almost complex structure on $TB$ defined by setting
\begin{equation}\label{cmplxstrTB}
    J\bigg(\dfrac{\del}{\del r_i}\bigg) = \dfrac{\del}{\del \alpha_i} \qquad J\bigg(\dfrac{\del}{\del \alpha_i}\bigg) = -\dfrac{\del}{\del r_i}
\end{equation}
and, since $B$ is affine, this almost complex structure is integrable, with holomorphic coordinates given by $r_j + i\alpha_j$, $j = 1,2,3$. The last thing we need to show is the construction of the lattice bundle $\Lambda_{TB}$. Observe that $TB \simeq \R^3 \times B$; in fact, $B$ is parallelizable, with a basis of global non-vanishing vector fields given by
\begin{equation}\label{GlobalVFonTB}
    E_0 = \dfrac{\del}{\del r_0} \qquad E_1 = e^{\lambda r_0}\dfrac{\del}{\del r_1} \qquad E_2 = e^{-\lambda r_0}\dfrac{\del}{\del r_2}.
\end{equation}
\begin{lemma}
    There exist a non-trivial well-defined lattice bundle $\Lambda_{TB} \subset TB$ defined by 
    \[\Lambda_{TB} := \text{span}_{\mathbb{Z}}\{E_0, p_{11}e^{-\lambda r_0}E_1 + p_{21}e^{\lambda r_0}E_2, p_{12}e^{-\lambda r_0}E_1 + p_{22}e^{\lambda r_0}E_2 \}\]
    where $E_0, E_1, E_2$ are the vector fields \eqref{GlobalVFonTB}. Moreover, $\Lambda_{TB} \simeq \mathbb{Z} \oplus P\mathbb{Z}^2$ in the affine coordinates $(\alpha_0, \alpha_1, \alpha_2)$.
\end{lemma}
\begin{proof}
    If we show well-definition, then by construction $\Lambda_{TB}$ is a non-trivial lattice bundle in $TB$ and in the affine coordinates it is, again by construction, given by $\mathbb{Z} \oplus P\mathbb{Z}^2$.\\[5pt]
    We first observe that it depends only on $r_0$, which is defined up to an integer translation. To show that $\Lambda_{TB}$ is well-defined it suffices to show that it is invariant under the transformation $r_0 \mapsto r_0 -1$. Under this transformation we get
    \[\Lambda_{TB} \mapsto \Lambda'_{TB} = \text{span}_{\mathbb{Z}}\{E_0, p_{11}e^{\lambda}e^{-\lambda r_0}E_1 + p_{21}e^{-\lambda}e^{\lambda r_0}E_2, p_{12}e^{\lambda}e^{-\lambda r_0}E_1 + p_{22}e^{-\lambda}e^{\lambda r_0}E_2 \}.\]
    Now, since
    \[PMP^{-1} = \begin{pmatrix}
        e^\lambda & 0\\
        0 & e^{-\lambda}
    \end{pmatrix}\]
    we have that
    \[\begin{pmatrix}
        e^\lambda p_{11} & e^\lambda p_{12} \\
        e^{-\lambda} p_{21} & e^{-\lambda} p_{22}\\
    \end{pmatrix} = PM = \begin{pmatrix}
        p_{11}m_{11} + p_{12}m_{21} & p_{11}m_{12} + p_{12}m_{22}\\
        p_{21}m_{11} + p_{22}m_{21} & p_{21}m_{12} + p_{22}m_{22}\\
    \end{pmatrix}\]
    and this tells us that
    \begin{align*}
    \Lambda'_{TB} = \text{span}_{\mathbb{Z}}\{E_0,(p_{11}m_{11} + p_{12}m_{21})e^{-\lambda r_0}E_1 + (p_{21}m_{11} + p_{22}m_{21})e^{\lambda r_0}E_2,\\
    (p_{11}m_{12} + p_{12}m_{22})e^{-\lambda r_0}E_1 + (p_{21}m_{12} + p_{22}m_{22})e^{\lambda r_0}E_2 \}
    \end{align*}
    which is the same span over $\mathbb{Z}$ as $\Lambda_{TB}$.
\end{proof}
Putting together all the previous computations we have proved the following. We denote $\Lambda_{TB}=\Lambda$ for simplicity in the following computations.
\begin{lemma}\label{IdentNwithTB}
    We can identify $N$ with $TB/\Lambda$ (endowed with the complex structure \eqref{cmplxstrTB}) via the biholomorphic map $TB/\Lambda \rightarrow N$ given by
    \[r_0 \mapsto x \qquad r_1 \mapsto y_1 \qquad r_2 \mapsto y_2\]
    \[\alpha_0 \mapsto y \qquad \alpha_1 \mapsto -x_1 \qquad \alpha_2 \mapsto -x_2.\]
\end{lemma}
Before moving to the symplectic side, we need to look at a holomorphic 3-form $\Omega$ on $N$. We take $\Omega$ to be, in the notations of Section \ref{Examples},
\[\check{\Omega} = i\varphi^{012}= i(dx + idy)\wedge(dx_1 + idy_1) \wedge (dx_2 + i dy_2) \]
and by the identification of Lemma \ref{IdentNwithTB} we can write it as
\begin{align*}
    \check{\Omega} &= i(dr_0 + id\alpha_0)\wedge(-d\alpha_1 + idr_1) \wedge (-d\alpha_2 + i dr_2) = \\
    & = (-d\alpha_0 + idr_0) \wedge(-d\alpha_1 + idr_1)\wedge (-d\alpha_2 + i dr_2).
\end{align*}
Therefore, if we perform the coordinate change 
\[\check{\theta}_i := -\alpha_i\]
we find 
\[\check{\Omega} = (d\check{\theta}_0 + idr_0) \wedge(d\check{\theta}_1 + idr_1)\wedge (d\check{\theta}_2 + i dr_2).\]
\begin{osservazione}
    The change of coordinates to $\check{\theta}_i$ does not affect the lattice neither the affine structure (as is does not involve the $r_i$ coordinates). Moreover, we can still identify holomorphically $N$ and $TB/\Lambda $.
\end{osservazione}
\subsection{The dual manifold}
To proceed with the construction we need to look at the dual manifold 
\[T^*B / \Lambda^*\]
where $\Lambda^*$ is the dual lattice to $\Lambda$ and it is therefore given by
\begin{equation}\label{DualLattice}
    \Lambda^* = (P^{-1})^T\mathbb{Z}^2 \oplus \tau^{-1}\mathbb{Z}.
\end{equation}
Being $\Lambda$ global and well-defined, $\Lambda^*$ inherits the same properties and so $T^*B / \Lambda^*$ a well-defined compact 6-dimensional manifold, which is endowed with a canonical symplectic structure constructed in the following way.\\[5pt]
Denote by $\theta_i$ the fiber coordinates on $T^*B / \Lambda^*$, dual to $\check{\theta}_i$. The basis coordinates are the same and therefore the canonical symplectic structure $\omega$ is given by
\begin{equation}\label{SymplecticStrMirror}
    \omega = d\theta_0 \wedge dr_0 + d\theta_1 \wedge dr_1 + d\theta_2 \wedge dr_2.
\end{equation}

\vspace{0.7cm}
The goal for finding the mirror of the Nakamura manifold is to identify a manifold $(M, \omega, \Omega)$ of type IIA, which can be simplectically identified with $T^*B/\Lambda^*$ and such that $\Omega$ is the \textit{Fourier-Mukai transform} of $\check{\omega}$, where $\check{\omega}$ is the balanced metric on $N$, which we recall is given by (in the new coordinates on $TB/\Lambda$)
\begin{equation}\label{BalancedMetricMirror}
    \check{\omega} = d\check{\theta}_0 \wedge dr_0 + e^{-2\lambda r_0}d\check{\theta}_1 \wedge dr_1 + e^{2\lambda r_0}d\check{\theta}_2 \wedge dr_2.
\end{equation}
We therefore compute the Fourier-Mukai transform of $e^{2\check{\omega}}$. 
\begin{lemma}\label{FMTransform}
    \begin{equation}
        \text{FT}(e^{2\check{\omega}}) = -\tau|\det P|(d\theta_0 + idr_0)\wedge (e^{\lambda r_0}d\theta_1 + ie^{-\lambda r_0}dr_1) \wedge (e^{-\lambda r_0}d\theta_2 + ie^{\lambda r_0}dr_2).
    \end{equation}
\end{lemma}
\begin{proof}
    By definition we need to compute 
    \[\text{FT}(e^{2\check{\omega}}) = \check{\pi}_*(\pi^*(\mathcal{P}\cdot e^{2\check{\omega}} \wedge e^{\sum d\check{\theta_i} \wedge d\theta_i}))\]
    where the maps $\pi, \check{\pi}$ are given by the fiber product
    \[\pi : (TB/\Lambda) \times (T^*B/\Lambda^*) \rightarrow TB/\Lambda\]
    \[\check{\pi} : (TB/\Lambda) \times (T^*B/\Lambda^*) \rightarrow T^*B/\Lambda^*\]
    and $\mathcal{P}$ is the polarization operator. Given the expression for $\check{\omega}$, the polarization operators reads
    \[\mathcal{P}\cdot e^{2\check{\omega}} = e^{i\check{\omega}}\]
    so that
    \begin{align*}
        &\pi^*(\mathcal{P}\cdot e^{2\check{\omega}} \wedge e^{\sum d\check{\theta_i} \wedge d\theta_i}) = \pi^*\exp(i\check{\omega} + \sum d\check{\theta_i} \wedge d\theta_i) \\
        &= \exp(d\check{\theta_0} \wedge (d\theta_0 + i dr_0) + \check{\theta_1} \wedge (d\theta_1 + ie^{-2\lambda r_0} dr_1)+  \check{\theta_2} \wedge (d\theta_2 + ie^{2\lambda r_0} dr_2).
    \end{align*}
    Now we can finish by expanding the exponential and integrating along the fibers of $\check{\pi}$. The only term we need to care about is the one involving $d\check{\theta_0} \wedge d\check{\theta}_1 \wedge d\check{\theta}_2$, which finally gives
    \[\text{FT}(e^{2\check{\omega}}) = -\tau|\det P|(d\theta_0 + idr_0)\wedge (e^{\lambda r_0}d\theta_1 + ie^{-\lambda r_0}dr_1) \wedge (e^{-\lambda r_0}d\theta_2 + ie^{\lambda r_0}dr_2)\]  
    where $\tau\det P$ is the volume of the fibers.
\end{proof}
\subsection{The mirror: the symplectic side}
With the Fourier-Mukai transform in hand, we can construct a manifold $M$ which will be the mirror of the the Nakamura manifold $N$. More precisely, $M$ will still be a family of Nakamura manifolds, constructed via a different lattice and identified simplectically with $T^*B/\Lambda^*$.\\[5pt]
Consider a matrix $M \in \SL(2; \mathbb{Z})$ as in the construction of the Nakamura manifold in Section \ref{Examples}. We start by observing that since $PMP^{-1} = D$ we get
\[(P^{-1})^T(M^{-1})^TP^T = D^{-1} \]
and still $(M^{-1})^T \in \SL(2, \mathbb{Z})$. Moreover, for any matrix of the form
\[S = \begin{pmatrix}
    0 & s_1 \\
    s_2 & 0
\end{pmatrix}\]
where $s_1, s_2 \in \{1,-1\}$, is such that $SD^{-1}S^{-1} = D$ and therefore
\[S(P^{-1})^T(M^{-1})^TP^TS^{-1} = D.\]
We can then consider the following construction. Take again the semi-direct product $\C \ltimes_{\rho} \C^2$ as in the construction of the Nakamura manifold in Section \ref{Examples} and take lattices
\[\Gamma_S = (\Z \oplus i\tau^{-1}\Z) \ltimes_{\rho} (P\Z^2 \oplus iSP^{-T}\Z^2)\]
where with $P^{-T}$ we have denoted the inverse transpose of $P$. The previous computations show the following.
\begin{lemma}
    There exist a well-defined family of compact 6-dimensional manifolds given by the quotients
    \[M_S := (\C \ltimes_{\rho} \C^2) \backslash \Gamma_S.
    \]
    All these manifolds carry a symplectic structure $\omega_S$ such that we get symplectomorphisms
    \[(M_S, \omega_S) \simeq (T^*B/\Lambda, \omega)\]
    where $\omega$ is given by \eqref{SymplecticStrMirror}.
\end{lemma}
\begin{proof}
    The first part is identical to the computations for the Nakamura manifold, which can be found in \cite{CT}. We indeed just need to check that $\Gamma_S$ are well-defined lattices i.e., they are preserved under the action of $\rho$. Since $\rho(i\tau^{-1}) = 1$, we are left to check that $\rho(1) = D$ preserves the second factor of $\Gamma_S$. This holds because the relation $S(P^{-1})^T(M^{-1})^TP^TS^{-1} = D$ is true.\\[5pt]
    Now, writing the coordinate $z \in \C$ resp., the coordinates $z_1,z_2 \in \C^2$ as $z = x+iy$ resp., $z_k = x_k + iy_k$, we get a global invariant co-frame of $(1,0)$-forms on $M_S$ given by
    \[ \theta^1 = dx \qquad \theta^4 = dy\]
	\[ \theta^2 = e^{-\lambda x}dx_1  \qquad \theta^5  = e^{-\lambda x}dy_1 \]
	\[ \theta^3 = e^{\lambda x}dx_2 \qquad \theta^6 = e^{\lambda x}dy_2\]
    which we notice is the same as for the Nakamura manifold $N$. We know the $M$ carries symplectic structures (see \cite{LT}) and we can take the symplectic form
    \[\omega_S = \theta^{14} + s_1\theta^{35} + s_2\theta^{26} = dx \wedge dy + s_1dx_2 \wedge dy_1 + s_2dx_1 \wedge dy_2. \]
    Under this symplectic structure we still have the foliation of Lagrangian tori $\mathcal{L}_{234}$, which has a special Lagrangian section coming from the same $L_{156} = B$ as before. Therefore, we again get a fibration in Lagrangian tori
    \[\pi : M_S \rightarrow B\]
    with the map $\pi$ explicitly given in the coordinates of $M$ by 
    \[\pi(x,x_1, x_2, y, y_1, y_2) = (x,y_1, y_2). \]
    We finally prove that $(M_S, \omega_S) \simeq (T^*B/\Lambda, \omega)$ under a symplectomorphism. This is determined in the following way. Again, denote the basis coordinates as $(x,y_1,y_2) = (r_0, r_1, r_2)$ and define new fiber coordinates
    \[\psi_0 := -y \qquad \psi_1 = s_1^{-1}x_2 \qquad \psi_2 = s_2^{-1}x_2.\]
    These coordinates have the following effects:
    \begin{itemize}
        \item The lattice in $B$ is not affected, while the lattice on the fibers is rescaled by 
        \[\begin{pmatrix}
            -1 & 0 \\
            0 & S^{-1}\\
        \end{pmatrix}\]
        which gives on the fibers a lattice isomorphic to $\tau^{-1}\Z \oplus P^{-T}\Z^{2} = \Lambda^{*}$.
        \item The symplectic form is written in these coordinates as
        \[\omega = d\psi_0 \wedge dr_0 + d\psi_1 \wedge dr_1 + d\psi_2 \wedge dr_2 \]
        where we have used the fact that since $s_1,s_2 \in \{1,-1\}$ then $s_i^{-1} = s_i$.
    \end{itemize}
    Combining these two facts we show that all the manifolds $(M_S, \omega_S)$ are symplectically identified with $(T^*B/\Lambda^*, \omega)$.
\end{proof}
\begin{osservazione}
    For all matrices $S$ as fixed we get symplectomorphic manifolds. Therefore, for simplicity we take $s_1 = s_2 = 1$ in the following.
\end{osservazione}
We are left with just the final step to mirror symmetry. Let us denote with $\Omega$ the Fourier-Mukai transform of $e^{2\check{\omega}}$ as in Lemma \ref{FMTransform}. We have seen that $M$ is identified with $T^*B/\Lambda^*$ by the coordinate change 
\[\psi_0 = -\theta_0 \qquad \psi_1 = \theta_2 \qquad \psi_2 = \theta_1.\]
Under this change $\Omega$ is written as
\[\Omega = -\tau|\det P|(d\psi_0 + idr_0) \wedge (e^{\lambda r_0}d\psi_2 + ie^{-\lambda r_0}dr_1) \wedge (e^{-\lambda r_0}d\psi_1 + ie^{\lambda r_0}dr_2) = \]
\[= i\tau|\det P|(\theta^1 + i\theta^4) \wedge (\theta^3 + i\theta^5)\wedge (\theta^2 + i \theta^6)\]
which is the 3-form such that $(M, \omega, \Omega)$ is a type IIA structure. We then have the following final result (which contains a small abuse of notation).
\begin{teorema}\label{MirrorNakamura}
   Denote $(\check{X},\check{\omega}, \check{\Omega}) = (N, \check{\omega}, \check{\Omega})$, $(X, \omega,\Omega) = (M, \omega, \Omega)$. Then $(\check{X},\check{\omega}, \check{\Omega})$ and $(X, \omega,\Omega)$ form a semi-flat mirror pair. 
\end{teorema}
\subsection{Cohomological computations}
Since $X$ and $\check{X}$ share a lot of similarities in this case, we can ask about their cohomologies and/or relations between them.\\[5pt]
Let us denote $\check{X}$ by $N$ as before. Since both $\check{X}$ and $X$ are constructed from the same Lie group, the are both compact solvmanifolds of \textit{completely solvable type}. Therefore, by Hattori's Theorem \cite{Hattori} we have $H^{\bullet}(\mathfrak{g})\simeq H^{\bullet}_{dR}(X;\R) \simeq H^{\bullet}_{dR}(\check{X};\R)$, that is the de Rham cohomology can be computed by invariant forms. In particular, the de Rham cohomology in this case does not "see" the difference in the lattice in the mirror pair. Therefore, we obtain the following table for the de Rham cohomology 
$$
\begin{array}{l}
H^0_{dR}(N;\R)=\R\langle 1\rangle \\[10pt]
H^1_{dR}(N;\R)=\R\langle \theta^1,\theta^4\rangle\\[10pt]
H^2_{dR}(N;\R)=\R\langle \theta^{14},\theta^{35},\theta^{26},\theta^{23},\theta^{56}\rangle\\[10pt]
H^3_{dR}(N;\R)=\R\langle \theta^{135},\theta^{126},\theta^{123},\theta^{156},\theta^{345},\theta^{246},\theta^{234},\theta^{456}\rangle\\[10pt]
H^4_{dR}(N;\R)=\R\langle \theta^{2356},\theta^{1246},\theta^{1345},\theta^{1456},\theta^{1234}\rangle\\[10pt]
H^5_{dR}(N;\R)=\R\langle \theta^{23456},\theta^{12356}\rangle\\[10pt]
H^6_{dR}(N;\R)=\R\langle \theta^{123456}\rangle\\[10pt]
\end{array}
$$
On the other hand, the \textit{Dolbeault cohomology} depends on the lattice and we detect a difference just by comparing the first Hodge number $ h^{1,0} $.\\ 
The construction of both $\check{X}$ and $X$ allows to apply Kasuya's results \cite{Kasuya} in the same way as in \cite[section 4.4]{CT}. Let us start with $\check{X} = N$. Directly by \cite[section 4.4]{CT} we have that $H^{1,0}(\check{X})$ is generated by the forms
\[ \varphi^{0} \qquad e^{-\frac{1}{2}\lambda(w - \bar w)}\varphi^1 \qquad e^{\frac{1}{2}\lambda(w - \bar w)}\varphi^2 \]
such that
\[ f(\gamma) = 1 \text{ for all } \gamma \in \Gamma' \]
where $ \Gamma'  = \Z \oplus i\tau\Z$ and $f(w) = e^{\pm\frac{1}{2}\lambda(w - \bar w)}$. The condition over $ f $ is equivalent to
\[\tau\lambda \in 2\pi\Z\]
therefore, we have
\[h^{1,0}(\check{X}) = \begin{cases}
    1 \quad \text{if } \tau\lambda \notin 2\pi\Z\\
    3 \quad \text{if } \tau\lambda \in 2\pi\Z\\
\end{cases}.\]
Applying the same method to $X$, we get instead the condition
\[\tau^{-1}\lambda \in 2\pi\Z\]
which gives
\[h^{1,0}(X) = \begin{cases}
    1 \quad \text{if } \tau^{-1}\lambda \notin 2\pi\Z\\
    3 \quad \text{if } \tau^{-1}\lambda \in 2\pi\Z\\
\end{cases}.\]
Already from here we can see that the conditions
\[ \tau\lambda \in 2\pi\Z \qquad \tau^{-1}\lambda \in 2\pi\Z \]
are not simultaneously satisfied, at least for generic values of $\tau$. In following, an explicit example.
\begin{esempio}
Take the Nakamura manifold build by the matrix
	\[ M = \begin{pmatrix}
		2 & 3 \\
		1 & 2 \\
	\end{pmatrix} \]
(which corresponds to the case of $ m = 2 $ in \cite[example 5.3]{CT}), so that $ \lambda = \log(2 + \sqrt{3}) $. Let us choose 
\[  \tau = \dfrac{2\pi}{\lambda} \in \R \setminus \{0\}. \]
With this choice we have of course $ \tau\lambda \in 2\pi\Z $, so that $ h^{1,0}(N) = h^{1,0}(\check{X}) = 3 $. On the other hand, we have
\[ \tau^{-1}\lambda = \dfrac{\lambda^2}{2\pi} = \dfrac{ \log(2 + \sqrt{3})^2}{2\pi} \notin 2\pi\Z \]
and therefore $ h^{1,0}(X) = 1 $.
\end{esempio}
These computations show that, even if the construction is totally similar, the Nakamura manifold and its semi-flat mirror \textit{are not generally biholomorphic}.\\[5pt]
As final cohomological point, we can in this case compute the \textit{refined symplectic Bott-Chern cohomology} explicitly. Consider the construction of the mirror Nakamura manifold (which for us is $X= M$, the symplectic side). Since the Lie algebra of the corresponding group is the same as for $N$, $(X, \omega)$ is a solvmanifold of \textit{completely solvable type} \cite[section 4.5]{CT}. Therefore, by \cite[theorem 3.2]{AK} (originally from \cite[theorem 3]{Macri}) the symplectic Bott-Chern cohomology is computed considering just \textit{invariant forms}. In this way, to compute the cohomologies we just need to look at the invariant forms with the right polarization i.e, $p \in \{2,3,4\}$ is the fiber index while $q \in \{1,5,6\}$ is the base index. The results are summarized in table \ref{tab:coomSympCSNakamura}; in the last column we can see appear the corresponding Hodge diamond, whose symmetric with respect to the middle line is the Hodge diamond of the \textit{refined Bott-Chern cohomology} of the Nakamura manifold $N$.
    \begin{table}[H]
        \centering
        \begin{tabular}{|c|c|c|}
        \hline
             $(p,q)$ & $H^{p,q}_{B, TY}(X)$ & $h^{p,q}_{B, TY}(X)$\\
             (0,0) & $\R\langle 1\rangle$ & 1 \\
             \hline
             (1,0), (0,1) & $\R\langle \theta^1\rangle$, $\R\langle \theta^4\rangle$ & 1,1 \\
             \hline
             (2,0), (1,1), (0,2) & $\R\langle \theta^{56}\rangle$, $\R\langle \theta^{14},\theta^{35},\theta^{26}\rangle$, $\R\langle\theta^{23}\rangle$ & 1,3,1 \\
            \hline
            (3,0), (2,1), (1,2), (0,3) & $\R\langle\theta^{156}\rangle$, $\R\langle \theta^{135},\theta^{126},\theta^{456}\rangle$, $\R\langle \theta^{123},\theta^{345},\theta^{246}\rangle$,  $\R\langle \theta^{234}\rangle$ & 1,3,3,1\\
            \hline
            (3,1), (2,2), (1,3) & $\R\langle\theta^{1456}\rangle$, $\R\langle \theta^{2356},\theta^{1246},\theta^{1345}\rangle$, $\R\langle \theta^{1234}\rangle$ & 1,3,1 \\
            \hline
            (3,2), (2,3) & $\R\langle \theta^{12356}\rangle$, $\R\langle \theta^{23456}\rangle$ & 1,1 \\
            \hline
            (3,3) & $\R\langle \theta^{123456}\rangle$ & 1 \\
            \hline
        \end{tabular}
        \caption{Refined symplectic Bott-Chern cohomology of $(X, \omega)$}
        \label{tab:coomSympCSNakamura}
    \end{table}

\section{Non-K\"ahler mirror symmetry: the complex parallelizable Nakamura manifold}\label{MirrorSymmCPNakamura}
\subsection{Constructions}
In the previous section we have computed a semi-flat mirror pair involving the \textit{completely solvable} Nakamura manifold. In particular, complete solvability "hides" any trivial evidence for $N$ and its mirror not to be diffeomorphic. To show that the situation is non-trivial (i.e., we can get non-K\"ahler semi-flat mirror pairs in which the mirror manifolds are not diffeomorphic) we also present the non-K\"ahler SYZ mirror of the \textit{complex parallelizable} Nakamura manifolds.\\[5pt]
Let us start by recalling the construction from \cite{Nakamura}. The start is the same as before, taking $M \in \SL(2;\Z)$ such that
\[PMP^{-1} = D = \begin{pmatrix}
    \lambda & \\
    & -\lambda\\
\end{pmatrix} \qquad \lambda \ne 0\]
for some $P \in \GL(2;\R))$. We can take a lattice $P\Z^2 \oplus iP\Z^2$ and have a complex 2-torus defined by
\[\mathbb{T}_{\C}^2 := \C^2 / (P\Z^2 \oplus iP\Z^2)\]
where the coordinates of $\C^2$ are denoted by $(z_2, z_3)$. Then the map
\[D[z_2,z_3] := [D(z_2, z_3)^t]\]
is a well-defined biholomorphism of $\mathbb{T}_{\C}^2$. Now we can consider $\C \times \mathbb{T}_{\C}^2$ and take the lattice $\Gamma = \langle T_1,T_2\rangle$, where the biholomorphisms $T_1, T_2$ of $\C \times \mathbb{T}_{\C}^2$ are defined by
\begin{align*}
    &T_1(z_1,[z_2,z_3]) := (z_1 +1 ,[e^{\lambda}z_2,e^{-\lambda}z_3]) \\
    &T_2(z_1,[z_2,z_3]) := (z_1 + i\tau , [ z_2 , z_3])
\end{align*}
for some $\tau \ne 0, \tau \in \R$. It is possible to prove that $\Gamma$ acts properly and discontinuously on $\C \times \mathbb{T}_{\C}^2$ and therefore we can define a compact 6-dimensional manifold
\begin{equation}\label{DefCPNakamura}
    \check{M} := \Gamma \backslash (\C \times \mathbb{T}_{\C}^2)
\end{equation}
which is called the \textit{complex parallelizable Nakamura manifold}.\\[5pt]
We can fix a complex structure on $\check{M}$ by fixing the following global co-frame of $(1,0)$-forms
\[\phi^1 = dz_1 \qquad \phi^2 = e^{-z_1}dz_2 \qquad \phi^3 = e^{z_1}dz_3\]
with structure equations
\begin{equation}\label{EqStrCPNakamuraComplex}
d\phi^1 = 0 \qquad d\phi^2 = -\phi^{12} \qquad d\phi^3 = \phi^{13}.
\end{equation}
As before, we fix a global real co-frame of 1-forms by
\[\phi^{k} = \theta^k + i\theta^{3+k} \qquad k = 1,2,3\]
with real structure equations (which we derive from \eqref{EqStrCPNakamuraComplex})
\begin{equation}\label{EqStrCPNakamuraReal}
\begin{aligned}
    &d\theta^1 = 0 \\
    &d\theta^2 = -\theta^{12} + \theta^{45}\\ 
    &d\theta^3 = \theta^{13} - \theta^{46}\\
    &d\theta^4 = 0\\
    &d\theta^5 = -\theta^{15} + \theta^{24}\\ 
    &d\theta^3 = \theta^{16} - \theta^{34}.
\end{aligned}
\end{equation}
If we denote by $\{E_1, ..., E_6\}$ the global real frame dual to $\{\theta^1, ..., \theta^6\}$, we get the following non-zero brackets
\begin{equation}\label{BracketsCPNakamura}
\begin{aligned}
    [E_1, E_2] = E_2 \qquad [E_1, E_3] = -E_3 \qquad [E_1, E_5] = E_5 \qquad [E_1, E_6] = -E_6\\
    [E_4, E_5] = -E_2 \qquad [E_4, E_6] = E_3 \qquad [E_2, E_4] = -E_5 \qquad [E_3, E_4] = E_6.
\end{aligned}
\end{equation}
Setting
\[\check{\omega} = \dfrac{i}{2}\sum_{k = 1}^{3}\phi^{k\bar k} ,\qquad \check{\Omega} = \phi^{123} \]
it is straightforward to check that $\omega$ is the fundamental form of a Hermitian metric on $M$, $\Omega$ is $(3, 0)$-form and the following system of equations
\begin{equation}\label{SystemTypeBCPNakamura}
    \begin{cases}
        d\check{\omega}^2 = 0 \\
        d\check{\Omega} = 0\\
        \check{\Omega} \wedge \overline{\check{\Omega}} = - iF\dfrac{\check{\omega^3}}{6} \\
        2i\del\delbar(F^{-1}\check{\omega}) = \rho_B
    \end{cases}
\end{equation}
endows $(\check{M}, \check{\omega}, \check{\Omega})$ with a Type IIB structure with $F = 8$ and $\rho_B = -\dfrac{1}{8}(\phi^{1\bar 12\bar2} + \phi^{1\bar 13\bar3})$. Also, these forms give $\check{M}$ the structure of a non-K\"ahler Calabi-Yau manifold. In particular, being $\{E_1, E_2, E_3\}$ involutive, Theorem \ref{DistribuzioniSingole} ensures that the foliation $\mathcal{L}_{123}$ is a special Lagrangian foliation.\\[5pt]
Now we consider the special Lagrangian $L_{0,0,0} \in \mathcal{L}_{123}$, which is given by
\[L_{0,0,0} = \{(x + \lambda a, e^{\lambda a}x_1 + p_{11}a_{1} + p_{12}a_2,e^{-\lambda a}x_2 + p_{21}a_{1} + p_{22}a_2, b\tau,p_{11}b_{1} + p_{12}b_2, p_{21}b_{1} + p_{22}b_2) | a_i,b_i \in \Z\}.\]
Reasoning in the same way as we did in Proposition \ref{Topology123And156} we get that $L_{0,0,0}$ is a compact special Lagrangian submanifold of $\check{M}$, from which we get a well-defined Lagrangian fibration
\[ T \hookrightarrow \check{M} \rightarrow B := L_{0,0,0}\]
where the fiber $T$ is a real Lagrangian  3-torus defined by $T = \R^3\backslash \Lambda$, for
\[ \Lambda := \tau\Z \oplus P\Z^2.\]
Therefore, we have all the tools to apply again the construction of supersymmetric non-K\"ahler mirror pairs. The symplectic side of the pair is in this case given by the manifold 
\begin{equation}\label{SymplecticSideMirrorNakamura}
	M := \Gamma' \backslash (\C \times \mathbb{T}^{2}_{\C})
\end{equation}
where $  \mathbb{T}^{2}_{\C} $ is the torus defined by 
\[\mathbb{T}_{\C}^2 := \C^2 / (P\Z^2 \oplus i(P^{-T})\Z^2)\]
and $\Gamma'$ is the lattice in $\C \times \mathbb{T}^{2}_{\C}$ (again considering coordinates $ (z_1, z_2, z_3) $ on $ \C^3 $) generated by the bioholomorphisms
\begin{align*}
	&T_1'(z_1,[z_2,z_3]) := (z_1 +\lambda , [e^{\lambda}z_2,e^{-\lambda}z_3]) \\
	&T_2'(z_1,[z_2,z_3]) := (z_1 + i\tau^{-1}, [z_2 , z_3])
\end{align*}
Keeping the same notations as in the construction of the complex side \eqref{DefCPNakamura}, a well-defined symplectic form $\omega$ is given by
\[\omega = \theta^{14} + \theta^{35} - \theta^{26}.\]
In this case we get the following result.
\begin{teorema}\label{TheoremMirrorCPNakamura}
    The symplectic side of the semi-flat non-K\"ahler supersymmetric mirror pair of the complex parallelizable Nakamura manifold $ (\check{M}, \check{\omega}, \check{\Omega}) $ is the symplectic manifold $(M, \omega, \Omega)$, where $M$ is defined as in \eqref{SymplecticSideMirrorNakamura} and $\Omega$ is the Fourier-Mukai transform of $e^{2\check{\omega}}$.
\end{teorema}

\subsection{Cohomological computations: a non-diffeomorphic mirror pair}
As we did in section \ref{MirrorSymmCSNakamura}, it is interesting to compute and compare different cohomologies of the semi-flat mirror pair. In this case, we observe peculiar behaviour that was not really considered in any known construction of non-K\"ahler mirror pairs.\\[5pt] 
In fact, recall that both the Nakamura manifold and its mirror belong to the class of \textit{completely solvable} solvmanifolds therefore, the de Rham cohomology only detects the Lie algebra behaviour. Since the Lie algebra of the complex and symplectic side of the pair is the same, the de Rham cohomology does not distinguish between $N$ and its mirror.\\[5pt]
The case of the complex parallelizable Nakamura manifold brings a substantial difference. Even if the geometric construction of the two sided of the mirror pair is the same, we do not have complete solvability. Indeed, the de Rham cohomology of both $ \check{M} $ and $ M $ is computed in \cite[table 13]{AK} and depends on the value of the parameter $\tau$ involved in the construction of the lattices. Since $ \check{M} $ is constructed using $\tau$, while $ M $ depends on $\tau^{-1}$ we obtain the following 
\begin{teorema}\label{deRhamCohomCPNakamura}
	Consider a semi-flat mirror pair $ (M, \check{M}) $ constructed from a parameter $\tau \in \pi\Z$. Then the Betti numbers of the pair are the ones apprearing in table \ref{tab:BettiNumbersMirrorPairCPNakamura}. Hence, $M$ and $ \check{M} $ are not diffeomorphic.
\end{teorema}
	  \begin{table}[htbp]
		\centering
		\begin{tabular}{|c|c|c|}
			\hline
			$k$ & $b_k(M)$ & $b_k(\check{M})$\\
			\hline
			$ 0 $ & $ 1 $ & $ 1 $ \\
			\hline
			$ 1 $ & $ 2 $  & $ 2 $ \\
			\hline
			$ 2 $ &$ 3 $ & $ 5 $ \\
			\hline
			$ 3 $ & $ 4 $& $ 8 $\\
			\hline
			$ 4 $ & $ 3 $& $ 5 $ \\
			\hline
			$ 5 $ & $ 2 $ & $ 2 $ \\
			\hline
			$ 6 $ & $ 1 $ & $ 1 $ \\
			\hline
		\end{tabular}
		\caption{Betti numbers of the mirror pair $ (M, \check{M}) $ for $\tau \in \pi\Z$}
		\label{tab:BettiNumbersMirrorPairCPNakamura}
	\end{table}
To conclude the analysis, we report the computation of the refined symplectic Bott-Chern cohomologies of the symplectic side $(M, \omega)$ of the mirror pair. The results are summarized in tables \ref{tab:coomSympCSPNakamuraCaseA} and \ref{tab:coomSympCSPNakamuraCaseB} (also in this case we have to distinguish the computation in the case of $\tau^{-1} \in \pi\Z$ or $ \tau^{-1} \notin \pi\Z $). The coordinates on the basis $ B $ are still denoted by $ r_i $, $ i = 0,1,2 $.
    \begin{table}[htbp]
	\centering
	\begin{tabular}{|c|c|c|}
		\hline
		$(p,q)$ & $H^{p,q}_{B, TY}(M)$ & $h^{p,q}_{B, TY}(M)$\\
		(0,0) & $\R\langle 1\rangle$ & 1 \\
		\hline
		(1,0), (0,1) & $\R\langle \theta^1\rangle$, $\R\langle \theta^4\rangle$ & 1,1 \\
		\hline
		(2,0), (1,1), (0,2) & 0, $\R\langle \theta^{14},\theta^{26}-\theta^{35}, e^{2r_0}(\theta^{26} + \theta^{35})\rangle$, 0 & 0,3,0 \\
		\hline
		(3,0), (2,1), (1,2), (0,3) & \text{all zero }& 0,0,0,0\\
		\hline
		(3,1), (2,2), (1,3) & 0, $\R\langle \theta^{2356},e^{-2r_0}(\theta^{1345} - \theta^{1246}), \theta^{1345} - \theta^{1246}\rangle$, 0 & 0,3,0 \\
		\hline
		(3,2), (2,3) & $\R\langle \theta^{12356}\rangle$, $\R\langle \theta^{23456}\rangle$ & 1,1 \\
		\hline
		(3,3) & $\R\langle \theta^{123456}\rangle$ & 1 \\
		\hline
	\end{tabular}
	\caption{Refined symplectic Bott-Chern cohomology of $(M, \omega)$ if $ \tau^{-1} \in \pi\Z $}
	\label{tab:coomSympCSPNakamuraCaseA}
\end{table}

    \begin{table}[htbp]
	\centering
	\begin{tabular}{|c|c|c|}
		\hline
		$(p,q)$ & $H^{p,q}_{B, TY}(M)$ & $h^{p,q}_{B, TY}(M)$\\
		(0,0) & $\R\langle 1\rangle$ & 1 \\
		\hline
		(1,0), (0,1) & $\R\langle \theta^1\rangle$, $\R\langle \theta^4\rangle$ & 1,1 \\
		\hline
		(2,0), (1,1), (0,2) & 0, $\R\langle \theta^{14},\theta^{35}-\theta^{26}\rangle$, 0& 0,2,0 \\
		\hline
		(3,0), (2,1), (1,2), (0,3) & \text{all zero}& 0,0,0,0\\
		\hline
		(3,1), (2,2), (1,3) & 0, $\R\langle \theta^{2356},\theta^{1246}-\theta^{1345}\rangle$, 0 & 0,2,0 \\
		\hline
		(3,2), (2,3) & $\R\langle \theta^{12356}\rangle$, $\R\langle \theta^{23456}\rangle$ & 1,1 \\
		\hline
		(3,3) & $\R\langle \theta^{123456}\rangle$ & 1 \\
		\hline
	\end{tabular}
	\caption{Refined symplectic Bott-Chern cohomology of $(M, \omega)$ if $ \tau^{-1} \notin \pi\Z $}
	\label{tab:coomSympCSPNakamuraCaseB}
\end{table}

\end{document}